\documentclass[11pt]{amsart}
\usepackage{amssymb, amsthm, amsfonts, mathdots}
\usepackage{epstopdf, graphics, xypic, colortbl, xcolor}

\newcommand{\ZZ}{\mathbb{Z}}
\newcommand{\RR}{\mathbb{R}}
\newcommand{\CC}{\mathbb{C}}
\newcommand{\PP}{\mathbb{P}}
\newcommand{\PGL}{\mathrm{PGL}}

\newcommand{\GL}{\mathrm{GL}}
\newcommand{\cG}{\mathcal{G}}
\newcommand{\cS}{\mathcal{S}}
\newcommand{\cT}{\mathcal{T}}

\newcommand{\cX}{\mathcal{X}}
\newcommand{\cY}{\mathcal{Y}}
\newcommand{\II}{\mathbb{I}}
\newcommand{\JJ}{\mathbb{J}}

\newcommand{\newword}[1]{\textbf{\emph{#1}}}
\newcommand{\Mzero}[1]{M_{0,#1}}
\newcommand{\Mbar}[1]{\overline{M}_{0,#1}}

\newcommand{\Sym}{\mathrm{Sym}}
\newcommand{\SYT}{\mathrm{SYT}}

\newcommand{\Fl}{\mathcal{F}\ell}
\newcommand{\tr}{\widetilde{r}}
\newcommand{\tII}{\widetilde{\II}}
\newcommand{\tg}{\widetilde{\gamma}}
\newcommand{\sh}{\mathrm{shape}}
\newcommand{\rsh}{\mathrm{rshape}}

\newcommand{\rect}{\scalebox{0.25}{\includegraphics{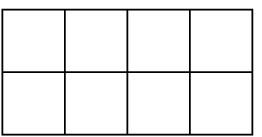}}}
\newcommand{\smallrect}{\scalebox{0.15}{\includegraphics{Rect.eps}}}

\newtheorem{Theorem}{Theorem}[section]
\newtheorem{theorem}[Theorem]{Theorem}
\newtheorem{prop}[Theorem]{Proposition}

\newtheorem{lemma}[Theorem]{Lemma}
\newtheorem{cor}[Theorem]{Corollary}

\theoremstyle{definition}
\newtheorem{remark}[Theorem]{Remark}

\newtheorem{example}[Theorem]{Example}

\title[Schubert problems from stable curves]{Schubert problems with respect to osculating flags of stable rational curves}
\author{David Speyer}

\begin{document}
\maketitle

\begin{abstract}
Given a point $z \in \PP^1$, let $F(z)$ be the osculating flag to the rational normal curve at point $z$.
The study of Schubert problems with respect to such flags $F(z_1)$, $F(z_2)$, \dots, $F(z_r)$ has been studied both classically and recently, especially when the points $z_i$ are real. 
Since the rational normal curve has an action of $PGL_2$, it is natural to consider the points $(z_1, \ldots, z_r)$ as living in the moduli space of $r$ distinct point in $\PP^1$ -- the famous $\Mzero{r}$. 
One can then ask to extend the results on Schubert intersections to the compactification $\Mbar{r}$.

The first part of this paper achieves this goal. We construct a flat, Cohen-Macaulay family $\cG(d,n)$ over $\Mbar{r}$, whose fibers over $\Mzero{r}$ are isomorphic to $G(d,n)$ and, given partitions $\lambda_1$, \dots, $\lambda_r$, we construct a flat Cohen-Macualay family $\cS(\lambda_1, \ldots, \lambda_r)$ over $\Mbar{r}$ whose fiber over $(z_1, \ldots, z_r) \in \Mzero{r}$ is the intersection of the Schubert varieties indexed by $\lambda_i$ with respect to the osculating flags $F(z_i)$. 

In the second part of the paper, we investigate the topology of the space $\cS(\lambda_1, \ldots, \lambda_r)(\RR)$ in the case that $\sum |\lambda_i| = \dim G(d,n)$. We show that $\cS(\lambda_1, \ldots, \lambda_r)(\RR)$ is a finite covering space of $\Mbar{r}$, and give an explicit CW decomposition of this cover whose faces are indexed by objects from the theory of Young tableaux.
\end{abstract}

\section{Introduction and Summary of Results}

Consider $2n-4$ points $z_1$, \dots, $z_{2n-4}$ in $\PP^1$. The \newword{Wronski problem} is to find degree $n-1$ rational functions $p(z)/q(z)$ on $\PP^1$ which have critical points at the $z_i$. 
Such a rational function is to be considered up to automorphisms of the target space. 
More intrinsically, let $V$ be the vector space of degree $n-1$ polynomials in $z$. Then $p(z)$ and $q(z)$ span a two dimensional subspace of $V$, and this subspace is an intrinsic characteristic of a the map $z \mapsto p(z)/q(z)$, unchanged by automorphisms of the target. 
If we consider $\langle p(z), q(z) \rangle$ as a point of the Grassmannian $G(2,n)$, then the condition that $p(z)/q(z)$ has a critical point at $z_i$ corresponds to a Schubert condition on $\langle p(z), q(z) \rangle$, and solving the Wronski problem amounts to finding the intersection of $2n-4$ Schubert conditions.

The group $\PGL_2$ acts on $\PP^1$, moving the points $z_i$.
This action extends to an action on $G(2,n)$, and to an action on the solutions of the Wronski problem.
So it is natural to consider the input to Wronski problem as $2n-4$ distinct points $z_i$ in $\PP^1$, modulo the action of $\PGL_2$. 
The space of $2n-4$ distinct points in $\PP^1$ modulo the action of $\PGL_2$ is $\Mzero{2n-4}$. 
Roughly speaking, this paper will extend this problem to the compactification $\Mbar{2n-4}$.

Our construction applies to any Grassmannian and any Schubert problem, not just the ones which arise from the Wronski problem.
To be precise, embed $\PP^1$ into $\PP^{n-1}$ by $\phi: t \mapsto (1:t:t^2:\cdots:t^{n-1})$.
Let $\phi^{(k)}(z)$ be the $k$-fold derivative of this map. Let $F_k(z)$ be the $k$-dimensional subspace of $\CC^k$ spanned by the vectors $\phi(z)$, $\phi^{(1)}(z)$, \dots, $\phi^{(k)}(z)$. If we reparametrize our source $\PP^1$, then the $\phi^{(k)}(z)$ change but the vector spaces $F_k(z)$ are unchanged. We write $F_{\bullet}(z)$ for the flag $0 \subset F_1(z) \subset F_2(z) \subset \cdots \subset F_{n-1}(z) \subset \CC^n$.
For a partition $\lambda$, let $\Omega(\lambda, z)$ be the subvariety of $G(d,n)$ which obey Schubert condition $\lambda$ with respect to the flag $F_{\bullet}(z)$. (We will review the definitions of Schubert conditions, and conventions for working with partitions, in Section~\ref{sec:notation}.)
For distinct points $z_1$, $z_2$, \dots, $z_r$ in $\PP^1$, and partitions $\lambda_1$, $\lambda_2$, \dots, $\lambda_r$, let $\Omega(\lambda_{\bullet}, z_{\bullet})$ be $\bigcap_{i=1}^r \Omega(\lambda_i, z_i)$. 

Let $U_r$ be the space of $r$ distinct ordered points in $\PP^1$ and let  $\lambda_1$, $\lambda_2$, \dots, $\lambda_r$ be the an $r$-tuple of partitions as above.
Let $\Omega(\lambda_{\bullet}) \subset U_r \times G(d,n)$ be the variety whose fiber over $(z_1, \ldots, z_r) \in U_r$ is $\Omega(\lambda_{\bullet}, z_{\bullet})$. 
This variety has been very well studied and the main results are the following:
For any $(z_1, \ldots, z_r) \in U_r$, the intersection $\bigcap \Omega(\lambda_i, z_i)$ is of the expected dimension $d(n-d) - \sum |\lambda_i|$~\cite[Theorem 2.3]{EH}. 
Combined with the obvious flatness of $\Omega(\lambda)$ over $U_r$, and the Cohen-Macaulayness of Schubert varieties~\cite{Ramanathan}, this means that 
$\Omega(\lambda_{\bullet})$ is Cohen-Macaulay, and is flat over $U_r$ of relative dimension $d(n-d) - \sum_{i=1}^r |\lambda_i|$. If $\sum |\lambda_i| = d(n-d)$, then the fiber of $\Omega(\lambda_{\bullet})$ over any point of $U_r(\RR)$ is a reduced union of real points \cite{MTV}. This last result was known as the \newword{Shapiro-Shapiro conjecture} before it was proved by Mukhin, Tarasov and Varchenko, and we will refer to it by this name within.
These are the results we will be generalizing to the case of stable curves.

The group $\PGL_2$ acts on $\PP^1$ and on $G(d,n)$; the latter action is by identifying $\CC^n$ with $\Sym^{n-1} \CC^2$ and taking the action on $G(d,n)$ induced by the action of $\GL_2$ on $\Sym^{n-1}(\CC^2)$. 
This gives rise to actions of $\PGL_2$ on $U_r$ and on $\Omega(\lambda_{\bullet}) \subset U_r \times G(d,n)$. For $r \geq 3$, the action of $\PGL_2$ on $U_r$ is free.
Free actions of reductive groups on quasi-projective varieties always have geometric quotients, so we may quotient by this action and get a family $\Omega(\lambda_{\bullet})/\PGL_2 \to U_r/\PGL_2 \cong \Mzero{r}$ which inherits the good properties of the family $\Omega(\lambda_{\bullet}) \to U_r$. 
Our goal will be to extend this family to $\Mbar{r}$. Here are our main results.

In all of the following theorems, let $0 \leq d \leq n$ and let $r \geq 3$. All partitions have at most $d$ parts, each of which are at most $n-d$.
Our theorems will take place within a flat, Cohen-Macaulay family $\cG(d,n)$ over $\Mbar{r}$ whose fibers over $\Mzero{r}$ are isomorphic to $G(d,n)$.
We will construct this family in Section~\ref{sec families}.

\begin{theorem} \label{thm Schub family}
For $\lambda$ a partition with at most $d$ parts, each at most $n-d$, there is a flat, Cohen-Macualay subfamily $\cS_i(\lambda)$ of $\cG(d,n)$, extending the family $\Omega(\lambda,z_i)/\PGL_2$ over $\Mzero{r}$.
For any partitions $\lambda_i$, the intersection $\bigcap_{i=1}^r \cS_i(\lambda_i)$ is Cohen-Macaulay and flat over $\Mbar{r}$ of dimension $d(n-d) - \sum |\lambda_i|$.
If $d(n-d) - \sum |\lambda_i| < 0$, then this intersection is empty.
\end{theorem}

We write $\cS(\lambda_{\bullet})$ for $\bigcap_{i=1}^r \cS_i(\lambda_i)$.
We now describe the irreducible components of the fibers of $\cS(\lambda_{\bullet})$. 
Let $(C, z_{\bullet})$ be a stable genus zero curve with $r$ marked points $z_1$, \dots, $z_r$. Let $C_1$, \dots, $C_v$ be the irreducible components of $C$ and
let $w_1$, \dots, $w_{v-1}$ be the nodes of $C$. 
We will say that $x \in C_j$ is a \newword{special point} if it is either a marked point, or a node.
For a partition $\mu= (\mu_1, \mu_2, \ldots, \mu_d)$, let $\mu^{C}$ be the \newword{complementary partition} $(n-d-\mu_d, n-d-\mu_{d-1}, \ldots, n-d-\mu_1)$.
Consider the following combinatorial data: For each component $C_j$ of $C$, and each special point $x \in C_j$, assign a partition $\nu(C_j, x)$, obeying the following conditions:
\begin{enumerate}
\item If $z_i$ lies in $C_j$, we have $\nu(C_j, z_i) = \lambda_i$.
\item If $C_j$ and $C_k$ meet at the node $w$, then $\nu(C_j, w) = \nu(C_k, w)^{C}$.
\end{enumerate}
Given the stable curve $(C, z_{\bullet})$ and this combinatorial data, let $\cT(C, z_{\bullet}, \nu_{\bullet})$ be the variety $\prod_{j=1}^{v} \Omega(\nu(C_j, x_{\bullet}), x_{\bullet})$, where $x_{\bullet}$ runs over the special points of $C_j$. Let $\cT(C, z_{\bullet}) = \bigsqcup_{\nu_{\bullet}} \cT(C, z_{\bullet}, \nu_{\bullet})$, where the union runs over all combinatorial data as above. 

\begin{theorem} \label{thm fibers}
The irreducible components of the fiber of $\cS(\lambda_{\bullet})$ over $(C, z_{\bullet})$ are isomorphic to the irreducible components of $\cT(C, z_{\bullet})$.
\end{theorem}
We do not attempt to describe how the components of $\cS(\lambda_{\bullet})$ meet.

For many possible maps $\nu$, the product $\prod_{j=1}^{v} \Omega(\nu(C_j, x_{\bullet}), x_{\bullet})$ is empty. 
In particular, if $\sum_{x \in C_j} |\nu(C_j, x)| > d(n-d)$ for some component $j$, then we get no components for this $\nu_{\bullet}$, so we can restrict to the case that  $\sum_{x \in C_j} |\nu(C_j, x)| \leq d(n-d)$ for every $C_j$.
In the case that $\sum |\lambda_i| = d(n-d)$, this can be made into an equality:
\begin{prop}\label{dimension equality}
If $\sum |\lambda_i| = d(n-d)$, then  $\prod_{j=1}^{v} \Omega(\nu(C_j, x_{\bullet}), x_{\bullet})$ is empty unless $\sum_{x \in C_j} |\nu(C_j, x)| = d(n-d)$ for every component $C_j$.
\end{prop}

One would like to say that we get an irreducible component of the fiber of $\cS(\lambda_{\bullet})$ for each $\nu_{\bullet}$ obeying $\sum_{x \in C_j} |\nu(C_j, x)| \leq d(n-d)$, but this isn't quite true. First of all, it may be that $ \Omega(\nu(C_j, x_{\bullet}), x_{\bullet})$ is empty even though the above inequality is satisfied.
If we restrict ourselves to those $\nu_{\bullet}$ where all the  $\Omega(\nu(C_j, x_{\bullet}), x_{\bullet})$ are nonempty, then $ \Omega(\nu(C_j, x_{\bullet}), x_{\bullet})$ is irreducible when $\sum_{x \in C_j} |\nu(C_j, x)| < d(n-d)$. So, if all the $\Omega(\nu(C_j, x_{\bullet}), x_{\bullet})$ are nonempty and all of the sums  $\sum_{x \in C_j} |\nu(C_j, x)|$ are $< d(n-d)$, then we do get one irreducible component.
However, if $\sum_{x \in C_j} |\nu(C_j, x)| = d(n-d)$, then $ \Omega(\nu(C_j, x_{\bullet}), x_{\bullet})$ consists of several points, the number of such points being given by the Littlewood-Richardson coefficient.

In particular, when $\sum |\lambda_i| = d(n-d)$ and we are over a real point of $\Mbar{r}$, we can make an especially nice statement.

\begin{theorem} \label{thm real family}
If $\sum |\lambda_i| = d(n-d)$, then the fiber of $\cS(\lambda_{\bullet})$ over $\Mbar{r}(\RR)$ is a reduced union of real points. In other words, $\cS(\lambda_{\bullet})(\RR)$ is an unbranched cover of $\Mbar{r}(\RR)$.
\end{theorem}

The preceding results are proved in Section~\ref{sec families}. The rest of the paper analyzes the topology of $\cS(\lambda_{\bullet})(\RR)$ when $\sum |\lambda_i| = d(n-d)$. 

There is a $CW$-decomposition of $\Mbar{r}(\RR)$, which we detail in Section~\ref{sec:CW}. The maximal faces are indexed by the $(r-1)!/2$  dihedral symmetry classes of circular orderings of $\{ 1, \ldots, r \}$.
The vertices of this $CW$-decomposition are indexed by trivalent trees whose leaves are labeled by $\{ 1,2, \ldots, r \}$, and correspond to stable curves whose dual graphs are those trees.

We will describe the pull-back of this $CW$-decomposition to the cover $\cS(\lambda_{\bullet})(\RR)$.
We will call a vertex of the $CW$-decomposition a \newword{caterpillar point} if every component of the curve contains at least one marked point (see Figure~\ref{cater example}).
See Sections~\ref{sec:growth} and~\ref{sec:degrowth} for the full definitions underlying the next theorems; 
readers experienced with tableaux combinatorics may be able to guess the definition of a cylindrical growth diagram by glancing at Figure~\ref{growth example}.

We first describe the case where all the partitions are a single box.

\begin{theorem}
Let $r=d(n-d)$ and let all the $\lambda_i$ be the partition $(1)$. We describe the pullback of the $CW$-decomposition of $\Mbar{r}(\RR)$ to $\cS(\lambda_{\bullet})(\RR)$. 
The maximal faces of this $CW$-decomposition lying over a given maximal face $\sigma$ of $\Mbar{r}(\RR)$ are indexed by cylindrical growth diagrams. 
Over a caterpillar point of $\Mbar{r}$, the vertices of $\cS(\lambda_{\bullet})(\RR)$ are indexed by standard Young tableaux of shape $(n-d)^d$. 
Let $\sigma$ be a maximal face of $\Mbar{r}$ and $\tau$ a maximal face of $\cS(\lambda_{\bullet})(\RR)$ lying above it; let $\gamma$ be the cylindrical growth diagram labeling $\tau$.
By following different paths through $\gamma$, we read off the tableaux labeling the vertices of $\tau$ over the caterpillar points of $\sigma$.
\end{theorem}

We also describe the adjacencies between these maximal faces (Proposition~\ref{wall cross}).

Transforming from one path to another in a growth diagram changes the tableau by promotion and other related operations.
Our result is thus related to results of Purbhoo~\cite{Purb1},~\cite{Purb2}. Purbhoo indexes fibers of $\Omega((1), (1), \cdots, (1))$ over certain points of $U_r$ by standard young tableaux and relates monodromy between these points to promotion.
Purbhoo works with $U_r$; we hope to show the reader that working over $\Mbar{r}$ provides cleaner results and proofs.

We now generalize this result to arbitrary partitions obeying $\sum |\lambda_i| = d(n-d)$. For a standard Young tableau $T$ of shape $(n-d)^d$, let $\alpha_s$ be the sub-skew tableau with entries $1+\sum_{i=1}^{s-1} |\lambda_i|$, $2+\sum_{i=1}^{s-1} |\lambda_i|$, \dots, $|\lambda_s|+\sum_{i=1}^{s-1} |\lambda_i|$.  
For another standard Young tableau $T'$ of the shape $(n-d)^d$, let $\alpha'_1$, \dots, $\alpha'_r$ be the corresponding sub-skew tableau. We will say that $T$ and $T'$ are dual equivalent if each $\alpha_i$ is dual equivalent to $\alpha'_i$. (This definition implicitly depends on the cardinality of the $\lambda_i$.)

\begin{theorem}
Let the $\lambda_i$ obey $\sum_{i=1}^r |\lambda_i| = d(n-d)$. We give a $CW$-decomposition of $\cS(\lambda_{\bullet})(\RR)$ which maps down to the $CW$-decomposition of $\Mbar{d(n-d)}(\RR)$. The points of $\cS(\lambda_{\bullet})(\RR)$ over caterpillar points of $\Mbar{r}$ are indexed by dual equivalence classes of standard Young tableaux for which $\alpha_i$ rectifies to a tableau of shape $\lambda_i$. The maximal faces of $\cS(\lambda_{\bullet})(\RR)$ lying over a maximal face $\sigma$ of $\Mbar{r}(\RR)$ are indexed by dual equivalence classes of cylindrical growth diagrams with shape $(\lambda_1, \ldots, \lambda_r)$. 
Let $\sigma$ be a maximal face of $\Mbar{r}$ and $\tau$ a maximal face of $\cS(\lambda_{\bullet})(\RR)$ lying above it; let $\gamma$ be the dual equivalence class of cylindrical growth diagrams corresponding to $\tau$. By following different paths through $\gamma$, we read off the dual equivalence classes of tableaux labeling the vertices of $\tau$ over the caterpillar points of $\sigma$.
\end{theorem}

We phrase the theorem here in a manner designed to get to the result quickly. 
In Section~\ref{sec:degrowth} we develop a theory of dual equivalence classes of cylindrical growth diagrams; we present and prove the result in that language.

I originally intended to conclude with an appendix discussing the relations between dual equivalence classes of growth diagrams, cartons \cite{TY}, hives and the octahedron recurrence \cite{KTW, HK} , but it has grown too lengthy and will be a separate paper.

\subsection{Acknowledgements}

This paper has developed over several years, and benefited from conversations with many mathematicians. Particularly, I would like to thank Allen Knutson, Joel Kamnitzer, Kevin Purbhoo and Frank Sottile.
Some of the research in this paper was done while I was supported by a research fellowship from the Clay Mathematical Institute.

%

\section{Notation and background} \label{sec:notation}

We write $[n]$ for $\{ 1,\ldots, n \}$. For any finite set $S$, we write $\binom{S}{d}$ for the subsets of $S$ of size $d$.
We will sometimes abbreviate $S \cup \{ i,j \}$ by $Sij$ or $S \setminus \{ k \}$ by $S \setminus k$.


We often want to work with partitions with at most $d$ parts, each at most $n-d$. The set of all such partitions will be denoted $\Lambda$, or $\Lambda(d,n)$ if necessary.
We denote the partition  $(0,0,0,\ldots)$ by $\emptyset$, $(1,0,0,\ldots)$ by $\square$ and $(n-d, n-d, \ldots, n-d)$ by $\rect$.
As mentioned in the introduction, set $(\lambda_1, \lambda_2, \ldots, \lambda_d)^C = (n-d-\lambda_d, \cdots, n-d-\lambda_2, n-d-\lambda_1)$.

For partitions $\lambda_1$, $\lambda_2$, \dots, $\lambda_s$ and $\mu$, with $|\mu| = \sum |\lambda_i|$, we write $c_{\lambda_1 \lambda_2 \cdots \lambda_s}^{\mu}$ for the Littlewood-Richardson coefficient.
We will often use the identity
$$c_{\lambda_1 \lambda_2 \cdots \lambda_r}^{\smallrect} = c_{\lambda_1 \lambda_2 \cdots \lambda_{r-1}}^{\lambda_r^C}.$$

For any two partitions $\mu$, $\nu$ with $\mu \subseteq \nu$, let $\SYT(\nu/\mu)$ be the set of chains of partitions which grow from $\mu$ to $\nu$ adding one box at a time.
For $T \in \SYT(\nu/\mu)$, we write $\sh(T) = \nu/\mu$.
If $\sh(T) = \lambda/\emptyset$, we say that $T$ \newword{is of straight shape} and we also write $\sh(T) = \lambda$.
We adopt the abbreviations $\SYT(\lambda)$ for $\SYT(\lambda/\emptyset)$, and $\SYT(d,n)$ for $\SYT(\rect)$.

Let $n-d \geq \lambda_1 \geq \lambda_2 \geq \cdots \geq \lambda_d \geq 0$ be a partition in $\Lambda(d,n)$, where we pad by zeroes as necessary. Then we define $I(\lambda)$ to be the subset $(\lambda_d+1, \lambda_{d-1}+2, \ldots, \lambda_1+d)$ of $[n]$. This is a bijection between $\Lambda$ and $\binom{[n]}{d}$.

For a flag $0=F_0 \subset F_1 \subset F_2 \subset \cdots \subset F_{n-1} \subset F_n = \CC^n$, with $\dim F_i =i$, and a partition $\lambda \in \Lambda$, we say that a $d$-plane $V$ in $\CC^n$ obeys the Schubert condition $\lambda$ with respect to $F_{\bullet}$ if $\dim (V \cap F_i) \geq \# (I(\lambda) \cap \{n-i+1, n-i+2, \ldots, n \})$.
The variety of $d$-planes obeying Schubert condition $\lambda$ with respect to $F_{\bullet}$ is denoted $\Omega(\lambda, F_{\bullet})$. 
It has codimension $|\lambda|$. The cohomology class $[\Omega(\lambda, F_{\bullet})] \in H^{2 |\lambda|}(G(d,n))$ is independent of $F_{\bullet}$, so we will drop the flag when speaking only of cohomology classes. 
We have
$$\prod_{i=1}^s [\Omega(\lambda_i)] = \sum c_{\lambda_1 \cdots \lambda_s}^{\mu} [\Omega(\mu)].$$
If $\sum |\lambda_i| = d(n-d)$ then $\prod [\Omega(\lambda_i)] $ is $c_{\lambda_1 \lambda_2 \cdots \lambda_r}^{\smallrect}$ times the fundamental class. If $F_1$, $F_2$, \dots, $F_r$ are generic flags then this intersection number is geometrically meaningful: $\bigcap_{i=1}^r \Omega (\lambda_i, F_i)$ consists of $c_{\lambda_1 \lambda_2 \cdots \lambda_r}^{\smallrect}$ reduced points.

For $z \in \mathbb{C}^1$, let $F_{\bullet}(z)$ be the flag spanned by the top rows of the matrix
$$\begin{pmatrix}
1 & z & z^2 & \cdots & z^n \\
0 & 1 & 2 z & \cdots & n z^{n-1} \\
0 & 0 &  2 & \cdots & n(n-1) z^{n-2} \\
& & & \ddots & \\
0 & 0 & 0 & \cdots & n! 
\end{pmatrix}.$$
We can extend this continuously to any $z \in \PP^1$ by letting $F_{\bullet}(\infty)$ be spanned by the top rows of
$$\begin{pmatrix}
0 & \cdots & 0 & 0 & 1 \\
0 & \cdots & 0 & 1 & 0 \\
0 & \cdots & 1 & 0 & 0 \\
&  \iddots & &  & & \\
1 & 0 & 0 & \cdots & 0
\end{pmatrix}.$$

These flags turn out to always be generic enough to give geometrically meaningful intersections:
\begin{prop}[\cite{EH}] \label{Schub trans}
For any distinct points $z_1$, $z_2$, \dots, $z_r$ in $\PP^1$, and any partitions $\lambda_1$, $\lambda_2$, \dots, $\lambda_r$, the intersection $\bigcap \Omega(\lambda_i, F_{\bullet}(z_i))$ has codimension $\sum |\lambda_i|$. If the product $\prod [\Omega(\lambda_i)]$ in $H^{\ast}(G(d,n))$ is zero, then this intersection is empty.
\end{prop}

We shorten $\Omega(\lambda, F_{\bullet}(z))$ to $\Omega(\lambda, z)$. We give an explicit formula for $\Omega(\square, z)$.

\begin{prop} \label{wronski equation}
Let $z \in \PP^1$ and let $(p_I)$ be a point of $G(d,n)$, given as a list of $\binom{n}{d}$ Pl\"ucker coordinates $p_I$. For $I \in \binom{[n]}{d}$, define 
$$\Delta(I) = \prod_{\substack{a,b \in I \\ a < b}} (b-a) \ \mbox{and} \ \ell(I) = \sum_{i \in I} i - \binom{d}{2}.$$ 
Then $(p_I)$ obeys the Schubert condition $\square$ with respect to $F_{\bullet}(z)$ if and only if
$$\sum_{I \in \binom{[n]}{d}} \Delta(I) \cdot P_I \cdot (-z)^{d(n-d)-\ell(I)}=0.$$
\end{prop}

\begin{proof}
See \cite[Proposition 2.3]{Purb1}. Purbhoo's conventions are related to ours by the involution $z \mapsto - z^{-1}$ of $\PP^1$. 
\end{proof}

Finally, we recall some tools for working with Cohen-Macaulay varieties.

\begin{theorem}[{\cite[Lemma 1]{Brion}}] \label{proper intersection}
Let $Z$ be smooth and let $X$ and $Y$ be subvarieties of $Z$ which are Cohen-Macaulay and pure of codimension $s$ and $t$ respectively. If $Y \cap Z$ has codimension $s+t$, then $Y \cap Z$ is Cohen-Macaulay.
\end{theorem}

\begin{theorem}[{\cite[Theorem 23.1]{Mats}}] \label{miracle flatness}
Let $B$ be smooth, $X$ Cohen-Macaulay, and $\pi: X \to B$ surjective with all fibers of the same dimension. Then $\pi$ is flat and the fibers of $\pi$ are Cohen-Macaulay.
\end{theorem}

We will often use these theorems together:
\begin{cor} \label{proper plus miracle}
Let $B$ be smooth, let $Z \to B$ be a smooth map, let $X$ and $Y$ be subvarieties of $Z$ which are Cohen-Macaulay, flat over $B$, and have codimensions $s$ and $t$ in $Z$.
If the fibers of $X \cap Y$ have codimension $s+t$, then these fibers are Cohen-Macaulay, and $X \cap Y$ is flat over $B$.
\end{cor}

\section{The families $\cG(d,n)$ and $\cS(\lambda_{\bullet})$} \label{sec families}

In this section, we construct the families $\cG(d,n)$ and $\cS(\lambda_{\bullet})$ over $\Mbar{r}$.  We fix $0 \leq d \leq n$. This construction will involve many varieties which are not seen again outside this section.

Let $T$ be a set with three elements. Let $\CC^2_T$ be the two-dimensional vector space of functions $f: T \to \CC$, modulo the subspace spanned by $(1,1,1)$. Let $\PP^1_T$ be the projective space of lines in $\CC^T$. Note that $\PP^1_T$ has three points $z(t)$ canonically labeled by $t \in T$, corresponding to the standard basis of $\CC^T$.
Let $V_T = \Sym^{n-1} \CC^2_T$ and let $G(d,n)_T$ be the Grassmannian of $d$-dimensional subspaces of $V_T$. 
We have the Segre embedding $\PP^1_T \hookrightarrow \PP(V_T)$, induced by sending $v \in \CC^2_T$ to $v^{\otimes (n-1)}$ in $V_T$. 
For $z \in \PP^1_T$, we can define a flag $F_{\bullet}(z)$ in $V_T$ as before, and we can then define Schubert varieties $\Omega(\lambda, z)_T \subseteq G(d,n)_T$ with respect to that flag. 
In particular, we write $\Omega(\lambda, t)_T$ for $\Omega(\lambda, z(t))_T$ for $t \in T$.

Let $Q=\{ p_1, p_2, q_1, q_2 \}$ be a set with four elements. Let $C$ be a genus zero curve with four distinct marked points $z(p_1)$, $z(p_2)$, $z(q_1)$, $z(q_2)$. Let $T_1 = (p_1, p_2, q_1)$ and $T_2 = (p_1, p_2, q_2)$. Then we have unique isomorphisms $\PP^1_{T_1} \cong C \cong \PP^1_{T_2}$, taking the labeled points of the $\PP^1_{T_i}$ to the labeled points of $C$. Let $\phi(C, z(p_1), z(p_2), z(q_1), z(q_2)) : \PP^1_{T_1} \to \PP^1_{T_2}$ be the composite isomorphism. 
This induces an isomorphism  $\phi(C, z(p_1), z(p_2), z(q_1), z(q_2))_{\ast} : G(d,n)_{T_1} \to G(d,n)_{T_2}$.
Let $\cX_0(p_1, p_2; q_1, q_2)$ be the closed subvariety of $\Mzero{Q} \times G(d,n)_{T_1} \times G(d,n)_{T_2}$ whose fiber over $(C, z(p_1), z(p_2), z(q_1), z(q_2))$ is the graph of $\phi(C, z(p_1), z(p_2), z(q_1), z(q_2))_{\ast}$. Let  $\cX(p_1, p_2; q_1, q_2)$ be the closure of  $\cX_0(p_1, p_2; q_1, q_2)$ in $\Mbar{Q} \times G(d,n)_{T_1} \times G(d,n)_{T_2}$.

For the next proposition and its proof, we abbreviate $\cX_0(p_1, p_2; q_1, q_2)$ and $\cX(p_1, p_2; q_1, q_2)$ to $\cX_0$ and $\cX$.

\begin{prop} \label{twocharts}
Over every point of $\Mzero{Q}$, the fiber of $\cX$ is isomorphic to $G(d,n)$.
Over the point of $\Mbar{Q}$ where $q_1$ and $q_2$ collide, the fiber is also isomorphic to $G(d,n)$.

Over the two points of $\Mbar{Q}$ where one of the $p$'s collides with one of the $q$'s, the fiber of $\cX$ is the reduced union of $\binom{n}{d}$ irreducible components, indexed by partitions $\nu \in \Lambda$. 
More specifically, when $p_1$ and $q_1$ collide, the fiber of $\cX$ is the reduced union $\bigcup_{\nu} \Omega(\nu, z(p_2)) \times  \Omega(\nu^{C}, z(p_1))$. 
When $p_1$ and $q_2$ collide, the fiber of $\cX$ is the reduced union $\bigcup_{\nu} \Omega(\nu, z(p_1)) \times  \Omega(\nu^{C}, z(p_2))$

The total space of $\cX$ is reduced and Cohen-Macaulay; $\cX \to \Mbar{Q}$ is flat.
\end{prop}

\begin{proof}
Over $\Mzero{Q}$, the fibers of $\cX$ are graphs of maps $G(d,n)_{T_1} \to G(d,n)_{T_2}$, so they are isomorphic to $G(d,n)$.
When $q_1$ collides with $q_2$, the isomorphism $\phi$ becomes the ``identity map'', taking the labeled points of $\PP^1_{T_1}$ to the labeled points of $\PP^1_{T_2}$, so the fiber of $\cX$ is again the graph of a map. Thus, all of the interest is near the two remaining boundary points of $\Mbar{Q}$. 

We switch to explicit coordinates. Let $z(p_1)$, $z(p_2)$ and $z(q_1)$ be at $0$, $\infty$ and $1$, and let $\tau$ be the position of $z(q_2)$. We will study what happens as $\tau \to 0$; the case where $\tau \to \infty$ is similar. 
The map $\PP^1_{T_1} \to \PP^1_{T_2}$ is given by $u \mapsto \tau u$.

Brion \cite[Section 2]{Brion} constructs a subvariety $\mathcal{Y}$ of $\CC^{n-1} \times G(d,n) \times G(d,n)$, flat over $\CC^{n-1}$, whose fiber over $(1,1,\ldots, 1)$ is the diagonal of $G(d,n) \times G(d,n)$ and whose fiber over $(0,0,\ldots, 0)$ is the union of Schubert varieties described above. 
Our family is precisely the restriction of Brion's family to the curve $\gamma \subset \CC^{n-1}$ parametrized by $(\tau,\tau^2, \cdots, \tau^{n-1})$. Since $\mathcal{Y}$ is flat, so is ours, and it has the described fibers.
The curve $\gamma$ is a complete intersection in $\CC^{n-1}$, cut out by the equations $x_k = x_1^k$, for $2 \leq k \leq n-1$. 
Since $\mathcal{Y}$ is flat over $\CC^{n-1}$, the regular sequence $(x_2-x_1^2, x_3-x_1^3, \ldots, x_{n-1}-x_1^{n-1})$ on $\CC^{n-1}$ pulls back to a regular sequence on $\mathcal{Y}$. So the quotient of a Cohen-Macaulay variety by a regular sequence is Cohen-Macaulay, so the restriction of$\mathcal{Y}$ to $\gamma$ is Cohen-Macaulay again.
\end{proof}

We will frequently engage in the following abuse of notation: With the above notation, let $Q \subset [r]$, and let $T_3$, $T_4$, \dots, $T_m$ be three element subsets of $[r]$.
Then we have a natural projection $\Mbar{r} \times \prod_{i=1}^{m} G(d,n)_{T_i} \to \Mbar{Q} \times G(d,n)_{T_1} \times G(d,n)_{T_2}$. 
We will use the notation $\cX(p_1, p_2; q_1, q_2)$ to also refer to the preimage of $\cX(p_1, p_2; q_1, q_2)$ inside  $\Mbar{r} \times \prod_{i=1}^{m} G(d,n)_{T_i}$. 
When necessary for clarity, we will say that we are considering $\cX(p_1, p_2; q_1, q_2)$ inside  $\Mbar{r} \times \prod_{i=1}^{m} G(d,n)_{T_i}$. 
Similarly, for $T \in \binom{[r]}{3}$ and $t \in T$, we will write $\Omega(\lambda, t)_T$ for the preimage of $\Omega(\lambda, t)_T$ under the map $\Mbar{r} \times \prod_{i=1}^{m} G(d,n)_{T_i} \to G(d,n)_T$.

We can now define the families $\cG(d,n)$ and $\cS(\lambda_{\bullet})$.
They are both defined as closed subfamilies of $\Mbar{r} \times \prod_{T \in \binom{[n]}{3}} G(d,n)_T$.
Namely, $\cG(d,n) = \bigcap_{p_1, p_2, q_1, q_2 \in [r]} \cX(p_1, p_2; q_1, q_2)$ and $\cS(\lambda_{\bullet}) = \cG(d,n) \cap \bigcap_{t \in T \in \binom{[r]}{3}} \Omega(\lambda_t, t)_{T}$. 
We recall the promises we have made about $\cG(d,n)$:
\begin{theorem}\label{thm Grass family}
The family $\cG(d,n)$ defined above is Cohen-Macaulay and flat of relative dimension $d(n-d)$ over $\Mbar{r}$. The fibers of $\cG(d,n)$ over $\Mzero{r}$ are all isomorphic to $G(d,n)$.
\end{theorem}

The remainder of this section consists of the proving that $\cG(d,n)$ and $\cS(\lambda_{\bullet})$ obey Theorems~\ref{thm Grass family}, \ref{thm Schub family} and \ref{thm real family}.
We begin by checking that the fibers over $\Mzero{r}$ are as claimed.

\begin{prop}
Over $\Mzero{r}$, the fibers of $\cG(d,n)$ are all isomorphic to $G(d,n)$, embedded in $\prod G(d,n)_T$ with the cohomology class of the small diagonal. 
The fiber of $\cS(\lambda_{\bullet})$ over $(z_1, z_2, \ldots, z_r) \in \Mzero{r}$ is isomorphic to $\Omega(\lambda_{\bullet}, z_{\bullet})$, inside the $G(d,n)$ of the previous sentence.
\end{prop}

\begin{proof}
Fix a point $(z_1, \ldots, z_r)$ to consider the fibers over. 
All of the varieties $\cX(p_1, p_2; q_1, q_2)$ are graphs of isomorphisms between two of the $G(d,n)_{T}$'s, and all these isomorphisms are compatible. So the fiber of $\cG(d,n)$ over $(z_{\bullet})$ are those points which are identified by all of these compatible isomorphisms, and is thus isomorphic to $G(d,n)$.
Imposing the Schubert conditions in addition imposes those Schubert conditions within this $G(d,n)$.
\end{proof}

All the claims in Theorems~\ref{thm Grass family}, \ref{thm Schub family} and \ref{thm real family}
are local on $\Mbar{r}$, so it is enough to prove them restricted to an open cover of $\Mbar{r}$.
We will now introduce terminology to describe the open cover.

Fix a trivalent tree $\Gamma$ with leaves labeled by $[r]$. We use $\Gamma$ to define an open subset $U(\Gamma)$ of $\Mbar{r}$. 
For a stable curve $(C, z_{\bullet})$, let $\Delta$ be the tree with leaves labeled by $[r]$ which encodes the structure of $C$. 
We define the point $(C, z_{\bullet})$ to be in $U(\Gamma)$ if there is a continuous map $\phi: \Gamma \to \Delta$ which takes leaves to leaves, preserving labels, takes internal vertices to internal vertices and takes edges either to edges or to internal vertices.
We note that the map $\phi$ is unique up to reparametrizing edges.

As $\Gamma$ ranges over trivalent trees with leaves labeled by $[r]$, the open sets $U(\Gamma)$ cover $\Mbar{r}$. 
Let $\cG(d,n)(\Gamma)$ and $\cS(\lambda_{\bullet})(\Gamma)$ be the open subsets of the corresponding families living over $U(\Gamma)$.
We write $\Gamma^{\circ}$ for the set of internal vertices of $\Gamma$.

For any $T \in \binom{[r]}{3}$, let $v(T)$ be the unique vertex of $\Gamma$ such that the elements of $T$ label leaves in three different components of $\Gamma \setminus v(T)$. 

For an internal vertex $v$ of $\Gamma$, compute the intersection inside $U(\Gamma) \times \prod_{T: v(T) = v} G(d,n)_T$ of those $\cX(p_1, p_2; q_1, q_2)$'s for which $v=v(p_1, p_2, q_1) = v(p_1, p_2, q_2) $. Call this intersection $G(d,n)_v$.

\begin{prop} \label{prop vertex reduction}
For every $T$ with $v(T)=v$, the projection $G(d,v)_v \to G(d,n)_T \times U(\Gamma)$ is an isomorphism.
\end{prop}

\begin{proof}
$G(d,n)_v$ is the intersection of many varieties of the form $\cX(p_1, p_2; q_1, q_2)$, each of which is the graph of an isomorphism between one $G(d,n)_T$ and another $G(d,n)_{T'}$, and all these isomorphisms are compatible.
\end{proof}

We will want to make computations in fibers of families over $U(\Gamma)$.
We now introduce suitable notation for this purpose.
Fix $u \in U(\Gamma)$. 
For $Z \to U(\Gamma)$, let $Z(u)$ be the fiber over $u$.
We also adopt the notation that, if $Z_1$, \dots, $Z_k$ are families over $U(\Gamma)$, then $\prod_{i=1}^k Z_i$ denotes the product over $U(\Gamma)$. For example,
$\prod_{i=1}^3 Z_i$ means $Z_1 \times_{U(\Gamma)} Z_2 \times_{U(\Gamma)} Z_3$.

Let $C$ be the corresponding stable curve and $\Delta$ the dual tree. Let $\Delta^{\circ}$ be the internal vertices of $\Delta$, corresponding to the irreducible components of $C$. 
For $v \in \Delta^{\circ}$, let $C(v)$ be the corresponding component of $C$.
Let $\CC^2(v)$ be a two-dimensional vector space with a chosen identification $\PP(\CC^2(v)) \cong C(v)$. 
Let $V(v) = \Sym^{n-1} \CC^2(v)$ and let $G(d,n)(v)$ be the Grassmannian of $d$-planes in $V(v)$. 

For an edge $e$ of $\Delta$, with endpoint $v \in \Delta^{\circ}$, let $p(v,e)$ be the marked point of $C(v)$ corresponding to $e$. 
Let $\Omega(\lambda, v,e)(u)$ be the $\lambda$-Schubert variety in $G(d,n)_v(u)$ corresponding the flag at $p(v,e)$.
Continuing our abuses of notation, we will also write $\Omega(\lambda,v,e)(u)$ for the preimage of this subvariety in any product  $\prod_i G(d,n)_{v_i}(u)$ where $v_1=v$.

We restate Proposition~\ref{twocharts} in the language we have developed.
\begin{prop} \label{two charts restated}
Let $p_1$, $p_2$, $q_1$ and $q_2 \in [r]$ with $v_i = v(p_1, p_2, q_i)$.

If $\phi(v_1) = \phi(v_2)$ then $\cX(p_1, p_2; q_1, q_2)(u)$ is the graph of an isomorphism between $\cG(d,n)_{v_1}(u)$ and $\cG(d,n)_{v_2}(u)$.

If $\phi(v_1) \neq \phi(v_2)$, let $e_i$ be the edge incident to $v_i$ which separates $v_i$ from $v_{3-i}$.
Then $\cX(p_1, p_2; q_1, q_2)(u) = \bigcup_{\nu \in \Lambda} \Omega(\nu, v_1, e_1)(u) \times \Omega(\nu', v_2, e_2)(u)$.
\end{prop}

We have the following corollaries:

\begin{cor} \label{prop edge reduction}
Let $(p_1, p_2, q_1, q_2)$ and $(p'_1, p'_2, q'_1, q'_2)$ be two quadruples of distinct elements of $[r]$, such that $v(p_1, p_2, q_1) = v(p'_1, p'_2, q'_1)$ and $v(p_1, p_2, q_2) = v(p'_1, p'_2, q'_2)$. Define these two vertices to be $v_1$ and $v_2$. Then $\cX(p_1, p_2; q_1, q_2)$ and $\cX(p'_1, p'_2; q'_1, q'_2)$ meet $G(d,n)_{v_1} \times_{U(\Gamma)} G(d,n)_{v_2}$ in the same subvariety.

Also, $\cX(p_1, p_2; q_1, q_2)$ and $\cX(p_1, p_2; q_2, q_1)$ meet $G(d,n)_{v_1} \times_{U(\Gamma)} G(d,n)_{v_2}$ in the same subvariety.
\end{cor}

\begin{proof}
From Proposition~\ref{twocharts}, all the schemes $\cX( \ , \ ; \ , \ )$ above are reduced, so we can check the equality on point sets.
So it is enough to show that, for every $u \in U(\Gamma)$, the fibers $\cX(p_1, p_2; q_1, q_2)(u)$ and $\cX(p'_1, p'_2; q'_1, q'_2)(u)$ are equal.
This is immediate from the description in Proposition~\ref{two charts restated}. 
The same argument shows $\cX(p_1, p_2; q_1, q_2) = \cX(p_1, p_2; q_2, q_1)$.
\end{proof}

So the following definition makes sense: For an unordered pair $\{ v_1, v_2 \} \in \Gamma^{\circ}$, we write  $\cX(v_1; v_2)$ for the subset of $G(d,n)_{v_1} \times_{U(\Gamma)} G(d,n)_{v_2}$ cut out by $\cX(p_1, p_2; q_1, q_2)$ for any $(p_1, p_2, q_1, q_2)$ with $v(p_1, p_2, q_1) = v_1$ and $v(p_1, p_2, q_2) = v_2$.  
For an edge $e$ connecting two internal vertices $v_1$ and $v_2$ of $\Gamma$, we write $\cX(e)$ for $\cX(v_1; v_2)$.

Let's sum up our progress.
We began with a subvariety of $U(\Gamma) \times G(d,n)^{\binom{r}{3}}$, cut out by $r(r-1)(r-2)(r-3)$ relations. 
However, thanks to Proposition~\ref{prop vertex reduction}, we can think of this as a subvariety of $U(\Gamma) \times G(d,n)^{r-2}$, where $r-2$ is the number of internal vertices of $\Gamma$. 
Proposition~\ref{prop edge reduction} has reduced us to $\binom{r-2}{2}$ relations: one for each pair $(v_1, v_2)$ of internal vertices of $\Gamma$.  Our eventual goal will be to reduce ourselves to $r-3$ relations: one for each internal edge of $\Gamma$.

For $w \in \Delta^{\circ}$ and $u \in U(\Gamma)$, let $G(d,n)_w(u)$ be $\bigcap_{\phi(v_1) = \phi(v_2)=w} \cX(v_1; v_2)(u)$, an intersection inside $\prod_{\phi(v) = w} G(d,n)_v(u)$. 
\begin{cor} \label{cor Delta vert reduction}
With the above notation, for any $v \in \Gamma^{\circ}$ with $\phi(v)=w$, we have the isomorphism $G(d,n)_v(u) \cong G(d,n)_w(u)$.
\end{cor}

\begin{proof}
As in the proof of Proposition~\ref{prop vertex reduction}, we just need to note that all of the $\cX(v_1; v_2)(u)$ all compatible isomorphisms.
\end{proof}

\begin{cor} \label{cor Delta edge reduction}
Let $u \in U(\Gamma)$. Let $w_1$ and $w_2$ be distinct vertices of $\Delta^{\circ}$. Let $v_1$, $v'_1$, $v_2$ and $v'_2$ be vertices in $\Gamma^{\circ}$ with $\phi(v_i) = \phi(v'_i) = w_i$. Then the identifications $G(d,n)_{v_i}(u) \cong G(d,n)_{v'_i}(u)$ for $i=1$, $2$, identifies $\cX(v_1; v_2)(u)$ with $\cX(v'_1; v'_2)(u)$.
\end{cor}

The above corollaries let us make definitions for $\Delta$ similar to those for $\Gamma$: We write $\cX(w_1; w_2)(u)$ for $\cX(v_1; v_2)(u)$ with $v_1$ and $v_2$ as in the above corollary.
For an edge $e$ connecting two internal vertices $w_1$ and $w_2$ of $\Delta$, we write $\cX(e)(u)$ for $\cX(w_1; w_2)(u)$.

\begin{prop} \label{edge intersection good}
The intersection $\bigcap_e \cX(e)$, where $e$ runs over the $r-3$ internal edges of $\Gamma$, is Cohen-Macaulay, reduced and flat of relative dimension $d(n-d)$ over $U(\Gamma)$.
\end{prop}

We'll write $\Gamma^{\circ}_{1}$ for the set of internal edges of $\Gamma$, and write $\Delta^{\circ}_1$ similarly.

\begin{proof}
By Proposition~\ref{prop vertex reduction}, $G(d,n)_v$ is a trivial bundle over $U(\Gamma)$ with fibers isomorphic to $G(d,n)$. 
So $\prod_{v \in \Gamma^{\circ}} G(d,n)_v$ is a trivial bundle with fibers $G(d,n)^{r-2}$. In particular, the fibers are smooth of dimension $(r-2) d(n-d)$.
And each $\cX(e)$ is Cohen-Macaulay of comdimension $d(n-d)$ (Proposition~\ref{twocharts}). 
In the next several paragraphs, our goal is to show that the fibers of $\bigcap_e \cX(e)$ are of expected dimension: namely,  $(r-2) d(n-d) - (r-3) d(n-d) = d(n-d)$. 

So, let $u \in U(\Gamma)$.
We use the notations $\Delta$, $\phi$ etcetera that we have associated with $u$ above.
Let $s = |\Delta^{\circ}|$.

First consider the intersection $\bigcap_{\phi\ \mbox{collapses}\ e} \cX(e)$, in the fiber over $u$, where we only intersect over edges $e$ of $\Gamma$ collapsed by $\phi$.
By Corollaries~\ref{cor Delta vert reduction} and~\ref{cor Delta edge reduction}, this intersection is $G(d,n)^s$ and, for $e$ an internal edge of $\Delta$, it makes sense to take about the subvariety $\cX(e)$ in this $G(d,n)^s$.
So, instead of studying the subvariety $\bigcap_{e \in \Gamma_1^{\circ}} \cX(e)$ of $G(d,n)^{r-2}$, we can study the  subvariety $\bigcap_{e \in \Delta_1^{\circ}} \cX(e)$ of $G(d,n)^{s}$.

Let $w \in \Delta^{\circ}$. Consider the intersection $\bigcap_{e \in \Delta^{\circ}_1} \cX(e)(u)$.
Let $e$ be an internal edge of $\Delta$ connecting $v_1$ and $v_2$. By Proposition~\ref{two charts restated}, 
$$\cX(e)(u) = \bigcup_{\nu \in \Lambda} \Omega(\nu, v_1, e)(u) \times \Omega(\nu^{C}, v_2, e)(u)$$

Let $N$ be the set of maps $\nu$  which, to each pair $(v,e)$ of edge $e \in \Delta^{\circ}_1$ and endpoint $v$ of $e$, assigns a partition $\nu(v,e)$ obeying the following condition:
For $v_1$ and $v_2$ the two endpoints of $e$, we have $\nu(v_1, e) = \nu(v_2, e)^{C}$.

Since we are working on the level of pointsets, intersection distributes over union and we deduce that
$$\bigcap_{e \in \Delta^{\circ}_1} \cX(e)(u) = \bigcup_{\nu \in N}  \bigcap_{e \in \Delta_1^{\circ}} \Omega(\nu(v_1, e), v_1, e)(u) \times \Omega(\nu(v_2, e), v_2, e)(u)$$
where $v_1$ and $v_2$ are the endpoints of $e$.
We can regroup the righthand side as
$$\bigcup_{\nu \in N} \prod_{v \in \Delta^{\circ}} \bigcap_{\substack{e \ni v  \\ e \in \Delta_1^{\circ}}} \Omega(\nu(v,e), v,e)(u).$$
Note that we have engaged in an abuse of notation, where $\Omega(\nu(v,e),v,e)(u)$ in the two displayed equations denote subvarieties of products $\prod G(d,n)_{v_i}(u)$ with different sequences $v_i$.

We will show that every nonempty term in the left hand union has dimension $d(n-d)$ as desired.
By Proposition~\ref{Schub trans}, the Schubert intersection $ \bigcap_{e \ni v} \Omega(\nu(v,e), v,e)(u)$ has codimension $\sum_{e \ni v} |\nu(v,e)|$.
So the dimension of the $\nu$-term is
$$s d(n-d) - \sum_{v \in \Delta^{\circ}} \sum_{\substack{e \ni v \\ e \in \Delta^{\circ}_1}} |\nu(v,e)| .$$
For each edge $e \in \Delta^{\circ}_1$, group together the two terms coming from $e$; their sum is $d(n-d)$. So the above expression is
$$s d (n-d) - |\Delta^{\circ}_1| d(n-d)   = \left( s - (s-1) \right) d (n-d) = d(n-d).$$
as desired.

We have now shown that the fibers of  $\bigcap_{e} \cX(e)$ are of the expected dimension.
By Corollary~\ref{proper plus miracle}, this shows that $\bigcap_e \cX(e)$ is Cohen-Macaulay and is flat over $U(\Gamma)$.
We still must show that $\bigcap_e \cX(e)$ is reduced. Cohen-Macaulay varieties, if generically reduced, are reduced, so we just need to show that $\bigcap_e \cX(e)$ is generically reduced.
So it is enough to check that, for $u \in U(\Gamma)  \cap \Mzero{r}$, the fiber $\bigcap_e \cX(e)(u)$ is reduced. This fiber is isomorphic to $G(d,n)$.
\end{proof}

 \begin{prop}
The intersection $\bigcap_{e} \cX(e)$, inside $\prod_{v \in \Gamma^{\circ}} G(d,n)_v$, is $\cG(d,n)$.
\end{prop}

\begin{proof}
As discussed above, we have $\cG(d,n) = \bigcap \cX(v; v')$ where $\{ v, v' \}$ runs over all unordered pairs of interval vertices of $\Gamma$. 
We want to show that $\bigcap \cX(v; v') = \bigcap \cX(e)$. The right hand side is the intersection of a smaller collection of varieties than the left hand side, so  
$\bigcap \cX(v; v') \subseteq \bigcap \cX(e)$ is obvious. To prove the reverse containment, we must establish that, for any two internal vertices $v$ and $v'$ of $\Gamma$, we have $\cX(v; v') \supseteq \bigcap \cX(e)$. 
We showed above that $\bigcap \cX(e)$ is reduced, so it is enough to check this containment on point sets.
Even better, we showed above that $\bigcap \cX(e)$ is flat over $U(\Gamma)$. So it is enough to check this containment over an open subset of $U(\Gamma)$.

We work over the open subset $U_0 = U(\Gamma) \cap \Mzero{r}$. Let $v=v_0$, $v_1$, $v_2$, \dots, $v'=v_{\ell}$ be the path through $\Delta$ from $v$ to $v'$. 
Over this open subsets, all of the families $\cX(v_{i-1}; v_i)$ are graphs of isomorphisms and the projection of $\bigcap \cX(v_{i-1}; v_i)$ onto $G(d,n)_{v} \times_{U_0} G(d,n)_{v'}$ 
is the composite isomorphism.
Moreover, this composite isomorphism is $\cX(v;v')$. So $\cX(v;v') \supseteq \bigcap_{i=1}^r \cX(v_{i-1}; v_i)$ over $U_0$, as desired.
\end{proof}
 
 We are now ready to prove the first of main Theorems:
 \begin{proof}[Proof of Theorem~\ref{thm Grass family}]
Fix a trivalent tree $\Gamma$. Over $U(\Gamma)$, we have shown that $\bigcap_{e \in \Gamma} \cX(e) = \cG(d,n)$, and that the left hand intersection is Cohen-Macaulay, reduced, and flat over $U(\Gamma)$.
Since the conditions of being Cohen-Macaulay, reduced, and flat are local on the base, and since the $U(\Gamma)$ cover $\Mbar{r}$, this shows that $\cG(d,n)$ is Cohen-Macaulay, reduced, and flat over $\Mbar{r}$.
 \end{proof}
 
We now move on to proving Theorem~\ref{thm Schub family}. 
Fix partitions $\lambda_1$, $\lambda_2$, \dots, $\lambda_r$ in $\Lambda$. Let $\Gamma$ be as before.

\begin{prop}
Let $t \in [r]$ and let $v \in \Gamma^{\circ}$ be the vertex adjacent to $t$. Let $p$, $q$ and $p'$, $q'$ be in $[r]$ so that $v(p,q,t) = v(p',q',t) = v$.
Then $\Omega(\lambda_t, t)_{p,q,t} \cap G(d,n)_v = \Omega(\lambda_t,t)_{p',q',t} \cap G(d,n)_v$ over $U(\Gamma)$.
\end{prop}

\begin{proof}
We may exchange labels so that $p$ and $p'$ are in one component of $\Gamma \setminus v$ and $q$ and $q'$ are in the other.
The isomorphism $\cX(p,t; q,q')$ equates $\Omega(\lambda_t, t)_{p,q,t}$ and $\Omega(\lambda_t,t)_{p,q',t}$; the isomorphism $\cX(q,t; p, p')$ equates $\Omega(\lambda_t,t)_{p,q',t}$ and $\Omega(\lambda_t,t)_{p',q',t}$.
\end{proof}

So it makes sense to talk about the Schubert variety $\Omega(\lambda_t, t)$ in $G(d,n)_v$ without talking about which $T$ we are using to define it.
And $\cS(\lambda_{\bullet})$ over $U(\Gamma)$ is $\cG(d,n) \cap \bigcap_{t \in [r]} \Omega(\lambda_t, t)$.

Let $u \in U(\Gamma)$. Reuse the notations $\Delta$ and so forth. 
We write $\Delta_1$ for the edge set of $\Delta$.

\begin{proof}[Proofs of Theorem~\ref{thm Schub family} and~\ref{thm fibers}]
As in the proof of Theorem~\ref{thm Grass family}, it is enough to check the claim over $U(\Gamma)$; all of the statements which follow take place over $U(\Gamma)$ (or specified subsets thereof).
The family $\prod_{v \in \Gamma^{\circ}} G(d,n)_v$ has dimension $(r-2) d(n-d)$ over $U(\Gamma)$.
The Cohen-Macaulay family $\cG(d,n)$ has codimension $(r-3) d(n-d)$; the Cohen-Macaulay family $\bigcap_{t \in [r]} \Omega(\lambda_t, t)$ has codimension $\sum_{t \in [r]} |\lambda_t|$. 
So, to prove that the intersection is Cohen-Macaulay and flat, we simply must check that all the fibers have the right dimension.

Fix $u \in U(\Gamma)$ so that we will compute the dimension of the fibers of $u$. As in the proof of Theorem~\ref{thm Grass family}, this fiber is 
$$\bigcup_{\nu \in N} \prod_{v \in \Delta^{\circ}} \left( \bigcap_{\substack{e \ni v \\ e \in \Delta_1^{\circ}}} \Omega(\nu(v,e), v,e)(u) \ \  \ \cap \! \! \! \bigcap_{\substack{t \in [r] \\ t \ \mbox{neighbors} \ v}} \Omega(\lambda_t, t)_v \right).$$

As in the proof of Proposition~\ref{edge intersection good}, this has dimension
$$s d(n-d) - \sum_{v \in \Delta^{\circ}} \sum_{\substack{e \ni v \\ e \in \Delta^{\circ}_1}} |\nu(v,e)| - \sum_{t \in [r]} |\lambda_t| = d(n-d) - \sum |\lambda_t|.$$
We now deduce the Cohen-Macaulayness and flatness as before.

What remains is to check reducedness of a generic fiber; this again follows by Proposition~\ref{Schub trans}.

In the course of our proof, we have also described the fibers of $\cS(\lambda_{\bullet})$ over $U(\Gamma)$. Our description is that of Theorem~\ref{thm fibers}.
\end{proof}

We also check Proposition~\ref{dimension equality}.
\begin{proof}[Proof of Proposition~\ref{dimension equality}]
Let $\sum_{t \in [r]} |\lambda_t| = d(n-d)$. Fix a stable curve $C$. Fix a map $\nu$ as in Theorem~\ref{thm fibers}. If $\nu$ contributes a nonempty variety to the above union, then every factor in the product over $v \in \Delta^{\circ}$ must be nonempty. So, for every vertex $v \in \Delta^{\circ}$, we have
$$\sum_{x \in C_j} |\nu(x, C_j)| =\sum_{\substack{e \ni v \\ e \in \Delta_1^{\circ}}} |\nu(v,e)|  + \sum_{\substack{t \in [r] \\ t \ \mbox{neighbors} \ v}} |\lambda_t| \leq d(n-d).$$
Summing over all $v$ as in the proof of Proposition~\ref{edge intersection good}, the above inequality implies that $\sum_{t \in [r]} |\lambda_t| \leq d(n-d)$. 
But our hypothesis is that $\sum |\lambda_t| = d(n-d)$, so the above inequality must be equality for every $v$.
\end{proof}

Finally, we move on to the claim about real points.

\begin{proof}[Proof of Theorem~\ref{thm real family}]
Let $\sum_{t \in [r]} |\lambda_t| = d(n-d)$. Let $u \in \Mbar{r}(\RR)$. We must consider the fiber of $\cS(\lambda_{\bullet})$ over $u$.
This fiber is
$$\bigcup_{\nu \in N} \prod_{v \in \Delta^{\circ}} \left( \bigcap_{\substack{e \ni v \\ e \in \Delta_1^{\circ}}} \Omega(\nu(v,e), v,e)(u) \ \ \  \cap \! \! \! \bigcap_{\substack{t \in [r] \\ t \ \mbox{neighbors} \ v}} \Omega(\lambda_t, t)_v \right)$$
By Proposition~\ref{dimension equality}, we only get a nonempty contribution to the union when 
$$\sum_{x \in C_j} |\nu(C_j, x)| = d(n-d)$$
for every $v \in \Delta^{\circ}$. Then, by the Shapiro-Shapiro conjecture, the term in the large parenthesis is a reduced union of $\RR$ points. 
So the interior fiber of $\cS(\lambda_{\bullet})$ over $u$ is a reduced union of $\RR$ points, as desired.
\end{proof}

For the rest of the paper, we will be interested in the case where $\sum_{t \in [r]} |\lambda_t| = d(n-d)$, and will be primarily interested in the case of real points.
Let's emphasize what our results say in this case.
$\cS(\lambda_{\bullet})$ is a finite cover of $\Mbar{r}(\RR)$. 
For $u \in \Mbar{r}(\RR)$, let $\Delta$ be the tree for $u$ as before. There are finitely many points of $\cS(\lambda_{\bullet})$ above $u$; fix one such point $x$.
For every vertex $v \in \Delta^{\circ}$, and edge $e \in \Delta_1$ incident to $e$, the point $x$ assigns a partition $\nu(e,v)$ to $(v,e)$. When $e \in \Delta^{\circ}_1$, we have $\nu(v,e) = \nu'(v,e)$; when $e$ joins $v$ to the leaf $t$, we have $\nu(e,v) = \lambda_t$. 
At every vertex $v \in \Delta^{\circ}$, we have $\sum_{e \ni v} |\nu(e,v)| = d(n-d)$.

\textbf{Remark:} It is straightforward to make analogous definitions for general $G/P$. The proof that $\cG(d,n)$ is Cohen-Macualay and flat over $\Mbar{r}$ generalizes straightforwardly. However, Proposition~\ref{Schub trans} fails \cite[Section~3.3.6]{SottExp}, so the theorems about $\cS(\lambda_{\bullet})$ may not generalize; there is much work to be done here.
In the next section, we will start proving results which are very special to Grassmannians.

\section{Monodromy along Curves}

Recall that we write $\square$ for the partition $(1)$. Our goal in this section is to understand the family $\cS(\lambda, \mu, \square, \square)$, where $|\lambda| + |\mu| = d(n-d)-2$. 
So $\cS(\lambda, \mu, \square, \square)$ is a finite cover of $\Mbar{4} \cong \PP^1$. 

We first recall (a special case of) Pieri's rule~\cite[Theorem 7.15.7]{EC2}:
\begin{prop} \label{Pieri box}
Let $|\lambda| + |\mu| = d(n-d)-1$. The Littlewood-Richardson coefficient $c_{\lambda \mu \square}^{\smallrect}$ is $1$ if $\mu^{C}$ is obtained by adding a single box to $\lambda$, and $0$ otherwise.
\end{prop}

\begin{cor} \label{Pieri two box}
Let $|\lambda| + |\mu| = d(n-d)-2$. The Littlewood-Richardson coefficient $c_{\lambda \mu \square \square}^{\smallrect}$ is $2$ if $\mu^{C}$ is obtained by adding a two nonadjacent boxes to $\lambda$; is $1$ if $\mu^{C}$ is obtained by adding a horizontal or vertical domino to $\lambda$, and is $0$ otherwise.
\end{cor}

\begin{proof}
We want to compute $c_{\lambda \square \square}^{\mu^{C}}$. By the associativity of the Littlewood-Richardson product, we can write
$$c_{\lambda \square \square}^{\mu^{C}} = \sum_{\kappa} c_{\lambda \square}^{\kappa} c_{\kappa \square}^{\mu^{C}}.$$
We can evaluate each term on the right hand side using Proposition~\ref{Pieri box}.
The only nonzero terms are where $\kappa$ is obtained by adding a single box to $\lambda$, and $\mu^{C}$ is obtained by adding a single box to $\kappa$.
So the sum is zero unless $\mu^{C}$ is obtained by adding two boxes to $\lambda$. If those two boxes are adjacent, there is one possible intermediate $\kappa$; if the two boxes are not adjacent, there are two possible terms.
\end{proof}

\begin{cor}
If $\mu^{C}$ does not contain $\lambda$, then $\cS(\lambda, \mu, \square, \square)$ is empty. If $\mu^{C} \setminus \lambda$ is a horizontal or vertical domino, then $\cS(\lambda, \mu, \square, \square) \cong \PP^1$, and the map to $\Mbar{4}$ is an isomorphism.
\end{cor}

\begin{proof}
The first sentence is immediate. The second is because a flat degree $1$ map to a normal variety must be an isomorphism.
\end{proof}

From now on, we assume that we are in the sole remaining case, where $\mu^{C} \setminus \lambda$ is two non-adjacent squares. Let $\kappa_1$ and $\kappa_2$ be the two partitions lying between $\lambda$ and $\mu^C$, with the added box of $\kappa_1$ to the southwest of the added box of $\kappa_2$. For example, we might have $\lambda = (3,1)$ and $\mu^{C} = (4,2)$, in which case $\kappa_1 = (3,2)$ and $\kappa_2 = (4,1)$. The next theorem summarizes our results.

\begin{theorem} \label{basic monodromy}
$\cS(\lambda, \mu, \square, \square)$ is a genus zero curve. The degree $2$ map from $\cS(\lambda, \mu, \square, \square)$ to $\Mbar{4}$ is branched over two non-real points.

As we travel around the circle $\cS(\lambda, \mu, \square, \square)(\RR)$, we encircle $\Mbar{4}$ twice, and thus the $6$ points of which lie over the boundary of $\Mbar{4}$ acquire a circular ordering. Each of these orderings corresponds to a stable curve with one node and that node is labeled by a pair of partitions $(\nu, \nu^{C})$ as discussed in the previous section. The circular ordering of the $\nu$'s is as follows:
\begin{enumerate}
\item $(\kappa_1, \kappa^{C}_1)$ 
\item $(\kappa_2, \kappa^{C}_2)$ 
\item $( (1,1), (1,1)^{C})$
\item $(\kappa_2, \kappa^{C}_2)$
\item $(\kappa_1, \kappa^{C}_1)$
\item $( (2), (2)^{C})$
\end{enumerate}
\end{theorem}

This result should be compared with Purbhoo's~\cite[Theorem 2.3]{Purb2}.

\begin{remark} \label{wimping out}
We do not actually use the knowledge of which of~(3) and~(6) is which. 
At the end of Section~\ref{sec:growth}, we will be able to give a more conceptual proof that~(3) and~(6) are positioned as claimed.
However, determining this is a straightforward though messy computation, given the other computations in this section, and it seemed silly not to include the answer.
Because we do not need the answer, we will get the reader started on this computation and leave the last steps to him or her.
\end{remark}

\begin{lemma} \label{quadratic monodromy}
Consider the conic $K$ given by $a u^2 - b u - c \tau u + d \tau =0$ with positive real coefficients $a$, $b$, $c$, $d$. If $bc> ad$, then the projective completion of $K(\RR)$ projects to the real $\tau$-line in an unbranched $2$ to $1$ cover.
\end{lemma}

\begin{proof}
We complete the $(\tau, u)$ plane to $\PP^2$, with coordinates $(\tau:u:1)$. 
The given conic corresponds to the symmetric form with matrix
$$\begin{pmatrix} 0 & -c/2 & d/2 \\ -c/2 & a & -b/2 \\ d/2 & -b/2 & 0 \end{pmatrix}$$
Since the determinant of this matrix is $-d/4(ad-bc) > 0$, and it has a $0$ on the diagonal, it must have signature $+\! -\! -$. 
Thus, the interior of $K(\RR)$ is the region where the quadratic is positive. In particular, the point with homogenous coordinates $(0:1:0)$ is within this interior.
Projection onto the $\tau$ line is projection from this point, so the mapping is unbranched.
\end{proof}

Figure~\ref{hyperbola} shows the conic $u^2 -2u - 2 \tau u + 3 \tau=0$, which we discuss in Example~\ref{g24 example}. (The asymptotes  and marked points of the figure will be explained as they become relevant.) The map to the $\tau$ line is vertical projection.

\begin{figure}
\centerline{\includegraphics{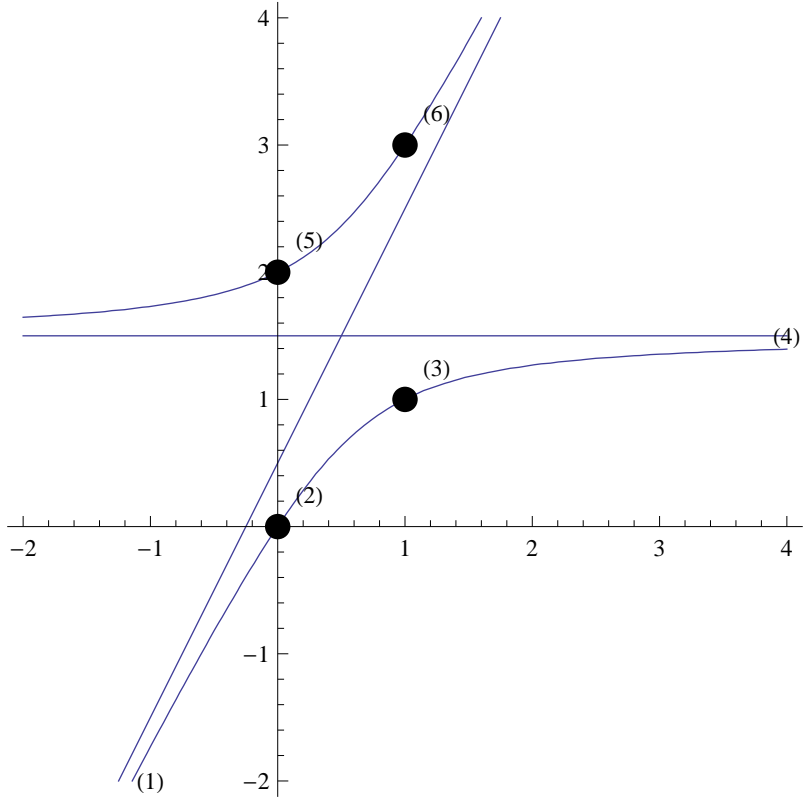}}
\caption{The conic $u^2 -2u - 2 \tau u + 3 \tau=0$ with its asymptotes and the $6$ points of Theorem~\ref{basic monodromy}} \label{hyperbola}
\end{figure}

Recall the notations $I(\lambda)$ and $\Delta(S)$ from Proposition~\ref{wronski equation} and the paragraphs preceding it.
Let $L = I(\lambda)$ and $M = I(\mu^{C})$. Then $L$ and $M$ are of the form $Sij$ and $S(i+1)(j+1)$ for some $S \in \binom{[n]}{d-2}$ and  some $i$, $j$; we have $i<i+1<j<j+1$ and all four of these indices are in $[n] \setminus S$. We have $I(\kappa_1) = S(i+1)j$ and $I(\kappa_2) = Si(j+1)$. Abbreviate these to $K_1$ and $K_2$.

\begin{lemma} \label{messy}
Let $\tau \in \CC \PP^1 \setminus \{ 0, 1, \infty \}$. Consider the Shapiro-Shapiro problem $(\lambda, \mu, \square, \square)$, with respect to the points $(0, \infty, 1, \tau)$. 
The only nonzero solutions of the Pl\"ucker coordinates are $p_{L}$, $p_{K_1}$, $p_{K_2}$ and $p_{M}$. Those values are
\[ \begin{array}{r r @{\cdot} l}
p_L =& \Delta(L)^{-1} & 1\\
p_{K_1} =& \frac{j-i-1}{j-i} \Delta(K_1)^{-1} & u\\
p_{K_2} =& \frac{j-i+1}{j-i}  \Delta(K_2)^{-1}  & (\tau u^{-1}) \\
p_{M} =& \Delta(M)^{-1} & \tau
\end{array} \]
where $(\tau,u)$ lie on the conic
\[ (j-i-1) u^2 - (j-i) u - (j-i) \tau u + (j-i+1) \tau =0 \]
\end{lemma}

\begin{example} \label{g24 example}
Consider the Schubert condition $\cS(\square, \square, \square, \square)$ in $G(2,4)$. 
Its solutions are given by
$$(p_{12}, p_{13}, p_{23}, p_{14}, p_{24}, p_{34}) = (0, \tau/2, u/2, \tau/(2u), 1/2, 0)$$
where $\tau$ and $u$ obey
$$u^2 - 2u - 2 \tau u + 3 \tau =0.$$
Observe that, considered as a quadratic in $u$, the discriminant of this equation is $4 (\tau^2-\tau+1)$, so there are two distinct roots of $u$ at all values of $\tau$ other than $(1\pm \sqrt{-3})/2$.
\end{example}

\begin{proof}
From the Pl\"ucker conditions at $0$ and $\infty$, we see that the only nonzero Pl\"ucker coordinates are $p_{L}$, $p_{K_1}$, $p_{K_2}$ and $p_{M}$. They obey the relation 
\begin{equation}
p_L p_M = p_{K_1} p_{K_2}, \label{P1}
\end{equation}
because there is a three term Pl\"ucker relation whose third term is $p_{Si(i+1)} p_{Sj(j+1)}$, and the variables in this last term are zero.

From proposition~\ref{wronski equation}, the quadratic
\[
p_{L} \Delta(L)  - (p_{K_1} \Delta(K_1) + p_{K_2} \Delta(K_2)) z + p_{M} \Delta(M) z^2
\]
has roots at $1$ and $\tau$. By rescaling the Pl\"ucker coordinates, we may assume that
\[
(p_{L} \Delta(L), p_{K_1} \Delta(K_1) + p_{K_2} \Delta(K_2), p_{M} \Delta(M)) = (\tau , 1+\tau , 1)
\]
Plugging into equation~(\ref{P1}), we have
\[ p_{K_1} p_{K_2} = \frac{\tau}{\Delta(L) \Delta(M)}. \]

We introduce the notation $\Delta(S, a)$ for $\prod_{s \in S} |a-s|$. So 
\[ \begin{array}{r r @{\cdot} l @{\cdot} l @{\cdot} l}
\Delta(L) = & (j-i) & \Delta(S) & \Delta(S, i) & \Delta(S, j)  \\
 \Delta(K_1) = & (j-i-1) & \Delta(S) & \Delta(S, i+1) & \Delta(S, j) \\
\Delta(K_2) = & (j-i+1) & \Delta(S) & \Delta(S, i) & \Delta(S, j+1)  \\
 \Delta(M) = &  (j-i) & \Delta(S) & \Delta(S, i+1) & \Delta(S, j+1)\\
\end{array} \]
and we have
\[ p_{K_1} p_{K_2} = \frac{\tau}{\Delta(S)^2 \Delta(S,i) \Delta(S,i+1) \Delta(S,j) \Delta(S, j+1) (j-i)^2} \]

So there is some $u$ with
\[ \begin{array}{r r @{\cdot} l l}
u =& \Delta(S) \Delta(S, i+1) \Delta(S,j) (j-i) & p_{K_1} =& \frac{j-i}{j-i-1} \Delta(K_1) p_{K_1} \\
\tau u^{-1} =&  \Delta(S) \Delta(S,i) \Delta(S, j+1) (j-i) & p_{K_2} =& \frac{j-i}{j-i+1} \Delta(K_2) p_{K_2}
\end{array} \]

We thus have
\[ p_{K_1} \Delta(K_1) + p_{K_2} \Delta(K_2) = \frac{j-i-1}{j-i} u + \frac{j-i+1}{j-i} \tau u^{-1} = 1 +\tau \]
or
\[ (j-i-1) u^2 - (j-i) u - (j-i) \tau u + (j-i+1) \tau =0. \]
All steps are trivially reversible.
\end{proof}

\begin{proof}[Proof of Theorem~\ref{basic monodromy}]
From Lemma~\ref{messy} we see that, for $\tau \in \Mzero{4}$, the variety $\cS(\lambda, \mu, \square, \square)$ is parametrized by the conic 
\[ (j-i-1) u^2 - (j-i) u - (j-i) \tau u + (j-i+1) \tau =0. \]
We denote the projective closure of this curve by $C$.
This is a hyperbola with a two asymptotes: A slant asymptote at $u = (j-1)/(j-i-1) \tau + 1/(j-i)(j-i-1)$, and a horizontal asymptote at $u=(j-i)/(j-i+1)$.
As it is a conic, $C$ has genus $0$. Moreover, since $(j-i-1)(j-i+1) = (j-i)^2-1 < (j-i)^2$, by Lemma~\ref{quadratic monodromy}, $C(\RR)$ is an unbranched double cover of $\Mzero{4}(\RR)$.

Since $\cS(\lambda, \mu, \square, \square)$ is closed in $\cG(d,n)$, and $C$ is normal, the map over $\Mzero{4}$ extends to a map $C \to \cS(\lambda, \mu, \square, \square)$. Since $\cS(\lambda, \mu, \square, \square)$ is flat over $\Mbar{4}$, the whole family $\cS(\lambda, \mu, \square, \square)$ is the closure of the part of the family over $\Mzero{4}$, 
and so the map  $C \to \cS(\lambda, \mu, \square, \square)$ is surjective. 
We now compute exactly what this map does over $0$, $1$ and $\infty$.

The points on $C$ over $\tau=\infty$ are the endpoints of the slant asymptote and the horizontal asymptote. 
The points on $C$ over $\tau=0$ are $(0,0)$ and $(0, (j-i)/(j-i-1))$; the points over $\tau = 1$ are $(1, 1)$ and $(1,(j-i+1)/(j-i-1))$. 
As we travel along $C(\RR)$ in the direction of increasing $\tau$, starting at the end point of the slant asymptote, the $u$ coordinate is continually increasing, which forces these points to be ordered as follows: The end of the slant asymptote, $(0,0)$, $(1,1)$, the end of the horizontal asymptote, $(0,(j-i)/(j-i-1))$, $(1,(j-i+1)/(j-i-1))$. 

These points are marked in Figure~\ref{hyperbola}. Points $(2)$, $(3)$, $(5)$ and $(6)$ are in the figure, and marked with solid dots; the points $(1)$ and $(4)$ are at the end of the asymptotes, out of the frame.

Plugging into Lemma~\ref{messy}, the corresponding Pl\"ucker coordinates are 
\[ \begin{array}{r@{} c@{\colon} c@{\colon} c@{\colon} c@{\colon} c@{\colon} c @{} l}
(& p_{12} & p_{13} & p_{23} & p_{14} & p_{24} & p_{34} &) = \\
(& 0 & 0 &  \ast & 0 & \ast & 0 &)\\
(& 0 & \ast & 0 & \ast & 0 & 0 &)\\
(& 0 & \Delta(L)^{-1} &  \frac{j-i-1}{j-i} \Delta(K_1)^{-1} & \frac{j-i+1}{j-i}  \Delta(K_2)^{-1}  & \Delta(M)^{-1} &0 &)\\
(& 0 &0  &0  & \ast & \ast & 0 &)\\
(& 0 & \ast & \ast & 0 & 0 & 0 &)\\
(& 0 &  \Delta(L)^{-1}  &  \frac{j-i+1}{j-i} \Delta(K_1)^{-1} &  \frac{j-i-1}{j-i}  \Delta(K_2)^{-1} &  \Delta(M)^{-1} & 0 &)\\
\end{array} \]
Here the $\ast$'s indicate nonzero numbers whose values will be irrelevant to the argument. 

Note that these are the Pl\"ucker coordinates on $G(d,n)_T$, where $T$ indexes the subset $(0,1, \infty)$ of $(0,\infty, 1, \tau)$. 
In particular the projection from $\cS(\lambda, \mu, \square, \square)$  to this Grassmannian does not identify the two points over any $\tau$, so the parametrization $C \to \cS(\lambda, \mu, \square, \square)$ is bijective. A bit more of a computation checks that the composite $C \to G(d,n)_T$ is also smooth (as opposed to introducing a cusp), so $C \to \cS(\lambda, \mu, \square, \square)$ is an isomorphism and $\cS(\lambda, \mu, \square, \square) \to G(d,n)_T$ is a closed embedding.

We now must check that the ordering above corresponds to the ordering listed in Theorem~\ref{basic monodromy}. 
The terms where $p_{K_1}$ (respectively $p_{K_2})$ vanish correspond to $\kappa_2$ (respectively $\kappa_1$) respectively.
The remaining computation is that $(0 : \Delta(L)^{-1} :  \frac{j-i-1}{j-i} \Delta(K_1)^{-1} :\frac{j-i+1}{j-i}  \Delta(K_2)^{-1}  : \Delta(M)^{-1} :0)$ corresponds to the partition $(1,1)$ and $(0 : \Delta(L)^{-1} :  \frac{j-i+1}{j-i} \Delta(K_1)^{-1} :\frac{j-i-1}{j-i}  \Delta(K_2)^{-1}  : \Delta(M)^{-1} :0)$ corresponds to the partition $(2)$. 
I have not found a slick way to finish this computation and therefore leave it to the reader. As discussed in Remark~\ref{wimping out}, we will see a cleaner proof later.
\end{proof}

\begin{remark}
Although $\cS(\lambda, \mu, \square, \square)$ turned out to embed into $G(d,n)_T$, one should note that simply writing down the relevant Schubert conditions in $G(d,n)_T \times \Mbar{4}$ does not give the correct variety over $0$, $1$ and $\infty$. For example, over $\tau=1$, simply imposing the condition $\square$ twice just writes down same condition twice: namely, that
$$p_{L} \Delta(L) -(p_{K_1} \Delta(K_1) + p_{K_2} \Delta(K_2)) z + p_{M} \Delta(M) z^2 $$
vanish at $z=1$. It is only by taking the closure of what happens at $\tau \neq 1$ that we obtain the stronger condition that $1$ be a double root of this polynomial.
\end{remark}

\subsection{An example in $\Fl(6)$}
We give an example in $\Fl(6)$ to demonstrate which statements generalize to $\Fl(n)$ and which do not.
We assume some familiarity with the cohomology ring of $\Fl(n)$.
We impose the Schubert class $153264$ with respect to $0$ and the Schubert class $514623$ with respect to $\infty$. 
This means that the components of the flag in $\Fl(6)$ are the row spans of the top justified $k \times 6$ submatrices of a matrix of the form
$$\begin{pmatrix}
\ast & \ast & 0 & 0 & 0 & 0 \\
0 & 0 & 0 & 0 & \ast & \ast \\
0 & 0 & \ast & 0 & 0 & 0 \\
\ast & \ast & 0 & 0 & 0 & 0 \\
0 & 0 & 0 & 0 & \ast & \ast \\
0 & 0 & 0 & \ast & 0 & 0 \\
\end{pmatrix}$$
The intersection of these Schubert varieties is isomorphic to $\PP^1 \times \PP^1$, and we will coordinatize it so that the top rows of the above matrix are 
$$\begin{pmatrix} 1 & u & 0 & 0 & 0 & 0 \\ 0 & 0 & 0 & 0 & 1 & v \end{pmatrix}.$$

Let $\square_d$ be the pull back to $H^2(\Fl(n))$ of the class $\square$ in $G(d,n)$.
We impose $\square_2$ at $1$ and $\square_3$ at $\tau$.
This gives
\[ \begin{array}{r l l l l}
4 & - 3 u & - 5v & +4 uv & =0 \\
16 & - 6 \tau u & - 30  \tau v & + 12 \tau^2 uv & =0 \\
\end{array} \]
Then $\tau$ and $v$ are related by 
$$(-12 +6 \tau) + (16  15  \tau-12 \tau^2) v + (- 30 \tau + 15 \tau^2)  v^2 =0$$
and $\tau$ and $u$ are related by a similar equation, giving a birational curve.

This curve has genus $1$, branched over the four roots of
$$256 - 960 \tau + 1281 \tau^2 - 720 \tau^3 + 144 \tau^4 =0.$$
These roots are at $0.678121$, $0.945553$, $1.41011$ and $1.96622$, so the analogue of the Shapiro-Shapiro conjecture fails.

In general, intersections over $\Fl(n)$ of the form $x \cdot y \cdot \square_d \cdot \square_e$ have cardinality $0$, $1$ or $2$. 
In the case where the cardinality is $2$, we have $yw_0 =(i j) (k \ell) x$ for some $i < j < k < \ell$. In our above example, $(i,j,k,\ell) = (1,2,5,6)$.
Computations as above gives a hyperelliptic curve branched over $2 (-i+j-k+\ell)$ points, and hence of arithmetic genus $-i+j-k+\ell-1$.
In some cases, this curve has a node at $\tau=1$ which is resolved in $\cS(x,y,\square_d, \square_e)$. In particular, this happens whenever $d=e$.
There are also examples where $\cS(x,y,\square_d, \square_e)$ actually has nodes.
As in the above example, the branch points can be real. However, if at least one of $j-i$ and $\ell-k$ is odd, then there is no branching over $(-\infty, -0)$.

\section{Real structure of $\Mbar{r}$} \label{sec:CW}

We review quickly the structure of $\Mbar{r}(\RR)$. An excellent introduction to this material is~\cite{Devad}.
Let $r \geq 3$.

$\Mbar{r}(\RR)$ is a connected compact manifold of dimension $r-3$. 
It has a natural regular $CW$ structure. The maximal faces are indexed by the $(r-1)!/2$ dihedral symmetry classes of ways to place the labels $[r]$ around a circle.
For such a circular ordering, the points in the interior of the corresponding face correspond to configurations of points on $\RR \PP^1$ with that circular order.

Each maximal face is isomorphic to the (simple) associahedron.
We should think of faces of the associahedron as indexing triangulations of the $r$-gon, where the \emph{edges} of the $r$-gon are labeled by the elements of $[r]$ in the corresponding circular order. 
Given a stable curve $C$, the nodes of $C$ correspond to the chords of the corresponding triangulation. If a chord separates the edges of the $[r]$-gon labeled by $A$ and the edges of the $[r]$-gon labeled by $[r] \setminus A$, then the corresponding node separates the points marked by $A$ and $[r] \setminus A$. 

When we cross a wall of the $CW$-complex, corresponding to single chord drawn in the $r$-gon, the circular orders on the two sides of the wall differ by reversing the ordering on one side of the chord.

Figure~\ref{M05} shows a portion of the CW structure on $\Mbar{5}(\RR)$.
The figure depicts $4$ maximal (two dimensional) faces, whose boundaries are dashed, surrounding a vertex.
We draw the topologies of the stable curves corresponding to the four maximal faces, to one of the edges and to the central vertex.

\begin{figure}
\centerline{\scalebox{0.7}{\includegraphics{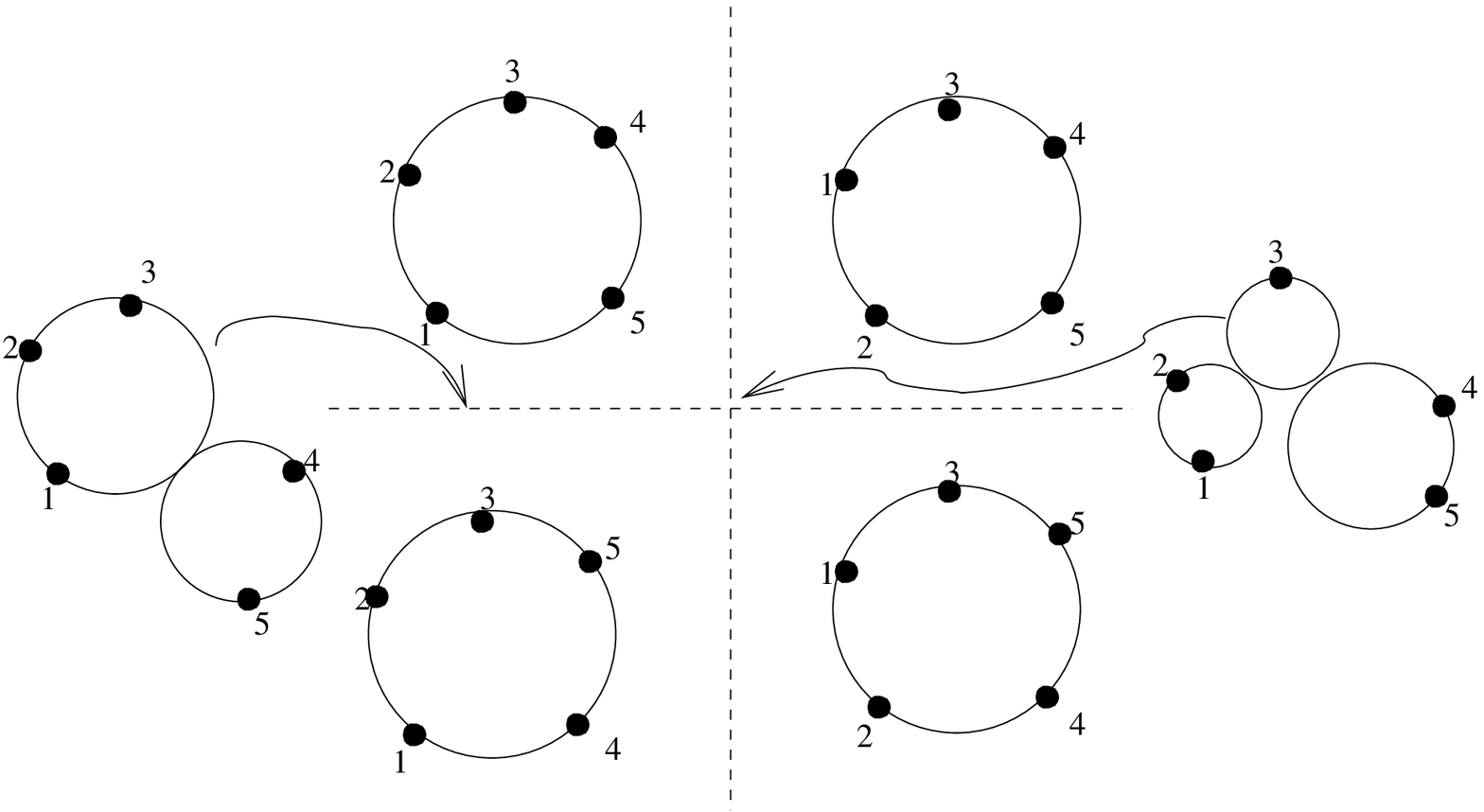}}}
\caption{A portion of $\Mbar{5}(\RR)$} \label{M05}
\end{figure}

\section{Cylindrical growth diagrams} \label{sec:growth}

Let $r=d(n-d)$ and suppose that all the partitions $(\lambda_1, \lambda_2, \cdots, \lambda_{r})$ are $\square$. 
In this section, abbreviate $\cS(\square, \square, \cdots, \square)$ to $\cS$. 
From Theorem~\ref{thm real family}, $\cS(\RR)$ is a covering of $\Mbar{r}(\RR)$, so we can pull back the regular $CW$-structure on $\Mbar{r}(\RR)$ to a regular $CW$-structure on $\cS(\RR)$.
Our aim in this section is to describe that structure.

Fix a circular order $( s(1), s(2), \cdots, s(d(n-d)))$ of $[r]$. We will consider the preimage in $\cS(\RR)$ of the corresponding face of $\Mbar{r}(\RR)$.
We make the definition that $[i,j]$ means the set $(s(i), s(i+1), \ldots, s(j-1), s(j))$, wrapping around modulo $r$ if necessary. We similarly use $[i,j)$, $(i,j]$ and $(i,j)$ to exclude one or both of the endpoints of this set.

Each maximal face of $\Mbar{r}$ is isomorphic to the associahedron, and its preimage in $\cS$ is the disjoint union of many\footnote{By the hook length formula~ \cite[Corollary 7.21.6]{EC2}, the degree of the cover is $r! / (d^{\underline{d}} (d+1)^{\underline{d}} \cdots (n-1)^{\underline{d}})$ where $x^{\underline{d}} = x(x-1)(x-2) \cdots (x-d+1)$.} copies of the associahedron. 
Let's fix one copy; call it $\sigma$.

Let $i$ and $j$ be distinct and circularly non-adjacent elements of $[r]$. 
Let $\sigma_{ij}$ be the facet of $\sigma$ corresponding to stable curves with two components, one containing $[i,j)$ and one containing $[j,i)$.
Consider any point of $\sigma$ on $\sigma_{ij}$.
There is a corresponding node of the stable curve. 
To each side of this node, the function $\nu$ assigns a partition. Let $\gamma_{ij}$ be the partition assigned to the side of the node containing $[j,i)$. 
(We emphasize: The side \emph{away} from $[i,j)$.)
So we have $\gamma_{ij} = \gamma^{C}_{ji}$.

\begin{lemma} \label{constant label}
The partition $\gamma_{ij}$ is independent of the choice of point on $\sigma_{ij}$.
\end{lemma}

\begin{proof}
Let $x$ and $z$ be two points on $\sigma_{ij}$. Let $X$ and $Z$ be the corresponding stable curves. Let $X_1$ (respectively $Z_1$) be the union of the components of $X$ (respectively $Z$) containing the points labeled by $[i,j)$; let $X_2$ and $Z_2$ contain the points labeled by $[j,i)$.
Define $Y$ to be the stable curve defined by gluing $X_1$ to $Z_2$ at the node of each, and let $y$ be the corresponding point of $\sigma_{ij}$. 

Consider any path in $\Mbar{j-i+1}(\RR)$ linking $X_2$ to $Z_2$, while maintaining the circular ordering of the points $[j,i)$. Glue the curves in this path to $X_1$, giving a curve in $\Mbar{r}(\RR)$. 
Since nothing is changing on the $X_1$ side of the node, the points of the $G(d,n)_v$'s for the vertices $v$ on that side of the node live in the same discrete set of Schubert solutions, everywhere along the path. 
By continuity, these points must simply remain constant as we travel along the path. In particular, the $\nu$-label of the node is the same at both ends of the path.

So $\gamma_{ij}$ is the same at $X$ and  $Y$. Similarly, $\gamma_{ij}$ is the same at $Y$ and  $Z$.
\end{proof}

So we may label each $\sigma$ by an array of partitions $\gamma_{ij}$. 
In this section, we will show that this array determines $\sigma$ and will give a description of which arrays occur. First, some combinatorial preludes.

Right now, our indices $i$ and $j$ live in $[r]$. 
It will be more convenient in the future to take our indexing set to be
$$\II : = \{ (i,j) \in \ZZ^2 : i \leq j \leq i+r \}$$
For $(i,j) \in \II$ with $2 \leq j-i \leq r-2$, we will write $\gamma_{ij}$ to mean the above defined $\gamma_{ij}$, with $i$ and $j$ reduced modulo $r$.
We set $\gamma_{kk} = \emptyset$, $\gamma_{k(k+1)} = \square$, $\gamma_{k(k+r-1)} = \square^{C}$ and $\gamma_{k(k+r)} = \rect$.

Define a \newword{cylindrical growth diagram} to be a map from $\II$ to $\Lambda$ (the set of partitions contained in $d^{n-d}$) obeying
\begin{enumerate}
\item $\gamma_{(i-1)j}$ and $\gamma_{i(j+1)}$ are obtained by adding a single box to $\gamma_{ij}$, and $\gamma_{ij} = |j-i|$. 
\item If the two boxes of the skew-shape $\gamma_{(i-1)(j+1)}  / \gamma_{ij}$ are nonadjacent, then $\gamma_{i(j+1)} \neq \gamma_{(i+1)j}$.
\end{enumerate}
The motivation for the term ``cylindrical'' will be explained after Proposition~\ref{glide reflect}.

Figure~\ref{growth example} shows $r+1=7$ rows of a growth diagram with $(d,n) = (2,5)$.
We omit the commas for brevity, writing $31$ rather than $(3,1)$.
Note that the bottom row is a repetition of the top; this is a general phenomenon that we will discuss in Proposition~\ref{glide reflect}.

Note that condition~(1) forces the boundary conditions $\gamma_{kk} = \emptyset$, $\gamma_{k(k+1)} = \square$, $\gamma_{k(k+r-1)} = \square^{C}$ and $\gamma_{k(k+r)} = \rect$. Condition~(2) is the growth diagram condition from~\cite[Proposition A.1.2.7]{EC2}. 

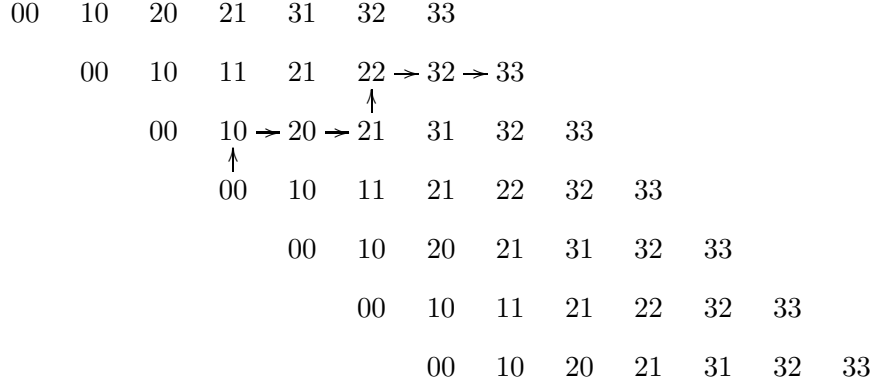
\begin{figure}
$$\xymatrix@=0.8 em{
00 & 10 & 20 & 21  & 31  & 32  & 33 &    &    &    &    &    & \\
   & 00 & 10 & 11  & 21  & 22\ar[r] & 32\ar[r] & 33 &    &    &    &    & \\
   &    & 00 & 10\ar[r] & 20\ar[r] & 21\ar[u] & 31 & 32 & 33 &     &    &    & \\
   &    &    & 00\ar[u] & 10 & 11 & 21 & 22 & 32 & 33 &    & &    \\
   &    &    &    & 00 & 10 & 20 & 21 & 31 & 32 & 33 &    &\\
   &    &    &    &    & 00 & 10 & 11 & 21 & 22 & 32 & 33 & \\
      &    &    &    & & & 00 & 10 & 20 & 21 & 31 & 32 & 33\\
}$$
\caption{A cylindrical growth diagram for $(d,n) = (2,5)$} \label{growth example}
\end{figure}

Define a \newword{path through $\II$} to be a sequence $(i_0, j_0)$, $(i_1, j_1)$, \dots, $(i_r, j_r)$ of elements of $\II$ such that $j_0=i_0$ and, for each $0 \leq k < r$, the difference $(i_{k+1}, j_{k+1}) - (i_k, j_k)$ is either $(-1, 0)$ or $(0,1)$. 
By reading a cylindrical growth diagram along a path through $\II$, we obtain a sequence of partitions, each of which is obtained by adding a single box to the previous one; 
in other words, we obtain a standard Young tableaux of shape $\rect$. In Figure~\ref{SYT example}, we translate the path in Figure~\ref{growth example} into a sequence of partitions and draw it as a Young tableaux in the standard way.
\begin{figure}
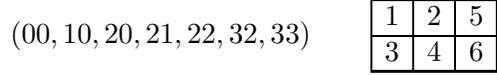

$$(00, 10, 20, 21, 22, 32, 33) \quad \quad \begin{array}{|c | c | c|}
\hline
{\LARGE 1} & {\LARGE 2} & {\LARGE 5} \\
\hline
{\LARGE 3} & {\LARGE 4} & {\LARGE 6} \\
\hline
\end{array}$$
\caption{A sequence of partitions, and its corresponding Young tableaux, obtained from the path in Figure~\ref{growth example}} \label{SYT example}
\end{figure}

\begin{lemma} \label{recursion}
Let $(i_0, j_0)$, $(i_1, j_1)$, \dots, $(i_r, j_r)$ be any path through $\II$, and fix values of $\gamma_{i_k j_k}$ in $\Lambda$ obeying that $\gamma_{i_{k+1} j_{k+1}}$ is obtained by adding a single box to $\gamma_{i_k j_k}$ and $|\gamma_{i_k j_k}| = j_k - i_k$. Then there is a unique way to extend $\gamma$ to a cylindrical growth diagram.
\end{lemma}

\begin{proof}
Condition~(2), combined with the boundary values $\gamma_{kk} = \emptyset$ and $\gamma_{k(k+r)} = \rect$ give a unique recursion.
\end{proof}

\begin{example}
If $\gamma_{ij} = (4,2,1)$, $\gamma_{(i-1)j} = (4,2,2)$ and $\gamma_{(i-1)(j+1)} = (4,3,2)$, then $\gamma_{i(j+1)}$ must be $(4,3,1)$ in order to be consistent with~(2).
If $\gamma_{ij} = (4,2,1)$, $\gamma_{(i-1)j} = (4,3,1)$ and $\gamma_{(i-1)(j+1)} = (4,4,1)$, then~(2) imposes no condition, but there is only one choice for $\gamma_{(i+1)j}$ obeying~(1), namely $(4,3,1)$.
\end{example}

\begin{lemma} \label{growth diagram relevant}
Given a maximal face $\sigma$ of the $CW$-structure on $\cS(\RR)$, define $\gamma_{ij}$ by the geometric recipe above. Then $\gamma_{ij}$ is a cylindrical growth diagram.
\end{lemma}

\begin{proof}
We first check the normalization that $|\gamma_{ij}| = |j-i|$. This is from Proposition~\ref{dimension equality}.

We next show that $\gamma_{i(j+1)}$ is obtained by adding a single box to $\gamma_{ij}$. 
Let $X$ be a stable curve with three components, one of which contains the marked point $j$ and the nodes joining those component with the other two.
The other components contain the marked points $[j+1, i)$ and $[i,j)$, in circular order. 
Then one node is labeled by $(\gamma_{ij}, \gamma^{C}_{ij})$ and the other is labeled by $(\gamma_{i(j+1)}, \gamma^{C}_{i(j+1)})$. 
The central component has labels $\gamma_{ij}$, $\gamma^{C}_{i(j+1)}$ and $\square$. 
So the Littlewood-Richardson coefficient $c^{\rect}_{\gamma_{ij}, \gamma^{C}_{i(j+1)}, \square}$ is positive or, equivalently, $c_{\gamma_{ij}, \square}^{\gamma_{i(j+1)}} >0$.
From the Pieri rule, this means that $\gamma_{i(j+1)}$ is obtained by adding a box to $\gamma_{ij}$.

The same argument shows that $\gamma_{(i-1)j}$ is obtained by adding a box to $\gamma_{ij}$.

Finally, we check condition~(2). 
Pick a curve $X$ with real points marked by $[i,j)$ in circular order, plus one more point between $j-1$ and $i$.
Pick another point $Z$ with real points marked by $[j+1, i-1)$ in circular order, plus one more point between $i-2$ and $j+1$. 
Glue $X$ and $Z$ to $\RR \PP^1$ at $0$ and $\infty$. Also mark two more points of this central $\RR \PP^1$, at $1$ and at $\tau \in (-\infty, 0)$, and label these points $j$ and $i-1$. As $\tau$ varies through $(-\infty, 0)$, this sweeps out a path through $\sigma$, and this path completes to a closed path for the limiting case $\tau = \infty$ and $\tau=0$.

On the central component, $0$ and $\infty$ are labeled by the partitions $\gamma_{ij}$ and $\gamma^{C}_{(i-1)(j+1)}$. 
The points $1$ and $\tau$ are labeled by $\square$. 
We want to compute the labels that wind up on the additional node formed when $\tau$ collides with $0$, and when $\tau$ collides with $\infty$. 
According to Theorem~\ref{basic monodromy}, they are unequal, as desired.
\end{proof}

The main result of this section is
\begin{theorem} \label{all square CW}
Over any maximal face $\sigma$ of $\Mbar{r}(\RR)$, the above construction gives a bijection between the preimage faces of $\cS(\RR)$ and cylindrical growth diagrams.
\end{theorem}

\begin{proof}[Proof of Theorem~\ref{all square CW}]
We recall the notation $(s(1), s(2), \ldots, s(r))$ for the circular ordering of $[r]$ corresponding to $\sigma$. We will use the notation $s(i)$ for $i$ an arbitrary integer, meaning to reduce $i$ modulo $r$.

Fix a path $\delta = ((i_0, j_0), (i_1, j_1), \ldots, (i_r, j_r))$ through $\II$.
As shown in Lemma~\ref{recursion}, there is a bijection between cylindrical growth diagrams, and maps $\delta \to \Lambda$ which start at $\emptyset$ and grow to $\rect$, adding one box at a time. 
Such a map $\delta \to \Lambda$ can, clearly, be thought of as a standard Young tableau of shape $\rect$; we recall the notation $\SYT(d,n)$ for the set of such standard tableaux.
So we must show that each standard Young tableau occurs for exactly one face of $\cS(\RR)$, over the fixed face of $\Mbar{r}(\RR)$.

Define a permutation $\pi$ of $[r]$ as follows: $\pi(k)$ is $s(i_k)$ if $i_k = i_{k-1} -1$ and $\pi(k) = s(j_{k-1})$ if $j_k = j_{k-1} + 1$. To see that this is a permutation, note that the quantity $j_k -i_k$ increases by $1$ every time $k$ increases. So, for each $k$, we add one more element to the interval $[i_k, j_k)$. When we reach $r$, we have $j_r - i_r = r$, so the interval $[i_r, j_r)$ will contain every equivalence class modulo $r$ exactly once. The permutation $\pi$ lists the elements of $[i_r, j_r)$ in the order they were added. In other words, $\pi$ is constructed so that $\pi([k]) = [i_k, j_k)$. 

We will now describe a particular stable curve $X$, and will write $x$ for the corresponding point in $\Mbar{r}(\RR)$. The reader may want to look at Figure~\ref{cater example}, where we show $X$ for the path in Figure~\ref{growth example}, and at the discussion after this proof.

The curve $X$ has $r-2$ components, $X_1$, $X_2$, \dots, $X_{r-2}$, with nodes connecting $X_{i}$ and $X_{i+1}$. 
The component $X_1$ has marked points $z(\pi(1))$ and $z(\pi(2))$; the component $X_{r-2}$ has marked points $z(\pi(r-1))$ and $z(\pi(r))$; for $2 \leq i \leq r-3$, the component $X_i$ contains the marked point $z(\pi(i+1))$. 
So the node between $X_{k-1}$ and $X_k$ separates the marked points labeled by $[i_k, j_k)$ and those labeled by $[j_k, i_k)$. 
Thus, knowing the values of $\gamma_{i_k j_k}$ is equivalent to knowing the partitions labeling the nodes in the fiber of $\cS$ above $x$.

From Theorem~\ref{thm fibers}, the fiber of $\cS$ over $x$ corresponds to all ways of labeling the nodes of $X$ by partitions, and of choosing a solution to the corresponding Schubert problem for each component of $X$.
In this case, each of these Schubert problems is of the form $\lambda \cdot \mu \cdot \square$. By the Pieri rule, $c_{\lambda \mu \square}^{\smallrect}$ is $1$ if $\mu^{C}$ is $\lambda$ with a box added on, and $0$ otherwise. 
So we only get a point of $\cS$ over $x$ if the labels of the nodes form a chain in $\SYT(d,n)$, and in this case we get one point. 
So we have a bijection from $\SYT(d,n)$ to the fiber of $\cS$ over $x$.

We now note that $x$ is a vertex of $\sigma$. For every node of $X$, the marked points lying to one side of this node form a cyclic interval in the circular ordering $s$, so the topology of $X$ occurs as a limit of the circular ordering topology given by $s$.

Fix a particular growth diagram $\gamma$ for which we want to show that there is a unique face labeled $\gamma$. 
Let $T$ be the standard young tableaux found by restricting $\gamma$ to the path $\delta$.

If $\tau$ is any maximal face of $\cS(\RR)$, lying above $\sigma$ and labeled by $\gamma$, then let $y$ be the point of $\tau$ above $x$. Then $y$ is labeled by $T$.
Conversely, if there is some $y$ above $x$ labeled by $T$, then, since $\cS(\RR) \to \Mbar{r}$ is a covering map, there is a unique maximal face $\tau$ of $\cS(\RR)$ lying above $\sigma$ and containing $y$.
Letting $\gamma'$ be the label of $\tau$, we see that $\gamma$ and $\gamma'$ coincide along the path $\delta$ and are hence equal, by Lemma~\ref{recursion}.

So maximal faces of $\cS(\RR)$ labeled by $\gamma$ are in bijection with points lying above $x$ labeled by $T$, and there is one of those.
\end{proof}

We deduce a combinatorial corollary, which is~\cite[Theorem 4.4]{Haiman}.

\begin{prop} \label{glide reflect}
In any cylindrical growth diagram, we have $\gamma_{ij} = \gamma_{(r+j) i}^{C} = \gamma_{(r+j)(r+j)}$
\end{prop}

\begin{proof}
We now know that every cylindrical growth diagram comes from an actual geometric face of $\cS$ lying over $\sigma$. 
Let $X$ be a curve in that face where one component contains the marked points $[i, j)$ and the other contains the marked points $[j,i) = [j, i+r)$. 
Then $\gamma_{ij}$ and $\gamma_{(r+j) i}$ are the two partitions labeling that node from opposite sides, so they are complementary, and $\gamma_{ij}$ and $\gamma_{(i+r)(j+r)}$ are two names for the partition labeling the same node.
\end{proof}

This now explains the terminology cylindrical: If we quotient $\II$ by $(i,j) \equiv (i+r, j+r)$, it is natural to draw the corresponding diagrams on a cylinder.

Figure~\ref{cater example} shows the caterpillar curve corresponding to the path in Figure~\ref{growth example}. The rows of Figure~\ref{growth example} are numbered $1$ through $7$, so that the path travels through positions $(4,4)$, $(3,4)$, $(3,5)$, $(3,6)$, $(2,6)$, $(2,7)$, $(2,8)$.
The marked points are shown as solid dots, and one side of each node is labeled with the partition coming from Figure~\ref{growth example}. (The complementary partition, labeling the opposite side of the node, has been omitted for clarity.)
\begin{figure}
\centerline{\includegraphics{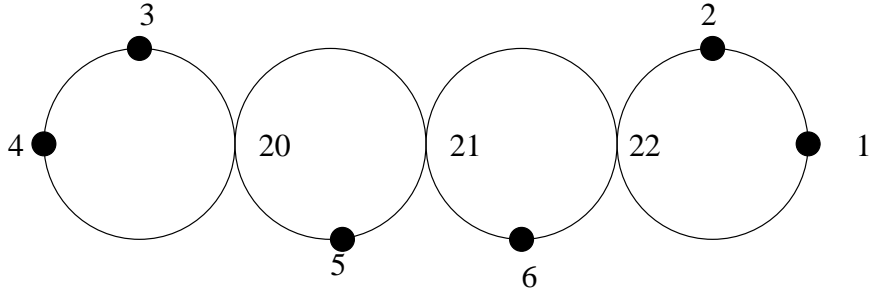}}
\caption{A caterpillar curve, corresponding to the path in Figure~\ref{growth example}} \label{cater example}
\end{figure}

\subsection{Effect of wall crossing}

We now understand how to index the maximal faces of $\cS$: They correspond to pairs of (1) a circular ordering of $[r]$, up to reflection and (2) a cylindrical growth diagram.
Note that, if we shift the circular ordering, we translate the cylindrical growth diagram, changing $(i,j)$ to $(i+k, j+k)$. 
If we reflect the circular order, we reflect the cylindrical growth diagram, changing $(i,j)$ to $(k-j, k-i))$.

We now discuss how these facets are glued together.
Let $p < q$ be distinct elements of $[r]$. 
Let $\hat{\sigma}$ be the adjacent face of $\Mbar{r}(\RR)$ where we reverse the order of $s(q)$, $s(q+1)$, \dots, $s(p-1)$ (indices are cyclic modulo $r$). 
Let $\gamma_{\bullet \bullet}$ be a cylindrical growth pattern indexing a face over $\sigma$, and let $\hat{\gamma}_{\bullet \bullet}$ index the face over $\hat{\sigma}$.

\begin{prop} \label{wall cross}
For $p \leq i \leq j \leq q$, we have $\gamma_{ij} = \hat{\gamma}_{ij}$. For $q \leq i \leq j \leq p+r$, we have $\gamma_{ij} = \hat{\gamma}_{(p+q-j-1)(p+q-i-1)}$.
\end{prop}

\begin{proof}
The nodes in question can occur on curves which live on the wall between $\sigma$ and $\hat{\sigma}$. 
The partition labeling the node is the same in both cases, but which element of $\II$ indexes that node will differ if $i$ and $j$ are in $[q,p)$, because the elements in $[q,p)$ are reversed when we cross the wall.
\end{proof}

The entries $\hat{\gamma}_{ij}$ for $(i,j)$ not in the given ranges are uniquely determined by the values within these ranges (Lemma~\ref{recursion}), but there is no simple description of them.

We could have equally well described crossing the wall by reversing $[p,q)$, in which case the circular growth diagram $\hat{\gamma}$ would differ by the appropriate reflection.
\begin{figure}
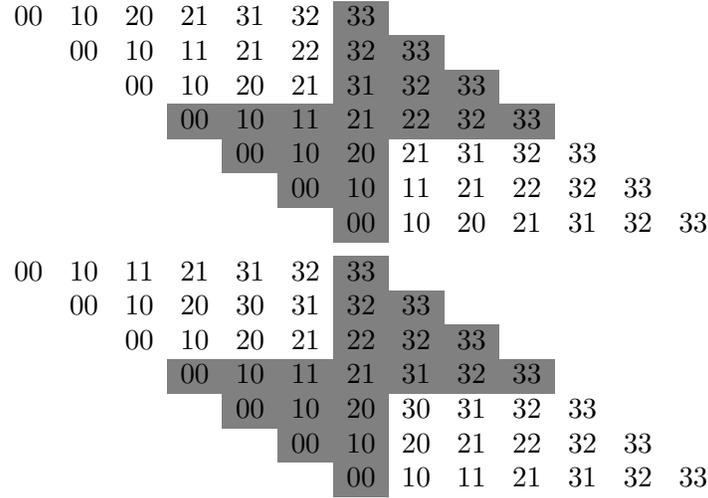

$$\begin{array}{ccccccccccccc}
00 & 10 & 20 & 21 & 31 & 32 & \cellcolor{gray} 33 &    &    &    &    &    & \\
   & 00 & 10 & 11 & 21 & 22 & \cellcolor{gray} 32 &\cellcolor{gray}  33 &    &    &    &    & \\
   &    & 00 & 10 & 20 & 21 & \cellcolor{gray} 31 &\cellcolor{gray}  32 & \cellcolor{gray} 33 &    &    &    & \\
   &    &    & \cellcolor{gray} 00 & \cellcolor{gray} 10 & \cellcolor{gray} 11 & \cellcolor{gray} 21 & \cellcolor{gray} 22 & \cellcolor{gray} 32 &\cellcolor{gray} 33 &    &    & \\
   &    &    &    & \cellcolor{gray} 00 &\cellcolor{gray}  10 &\cellcolor{gray}  20 & 21 & 31 & 32 & 33 &    & \\
   &    &    &    &    &\cellcolor{gray} 00 &\cellcolor{gray} 10 & 11 & 21 & 22 & 32 & 33 & \\
   &    &    &    & & & \cellcolor{gray} 00 & 10 & 20 & 21 & 31 & 32 & 33\\
\end{array}$$
$$\begin{array}{ccccccccccccc}
00 & 10 & 11 & 21 & 31 & 32 & \cellcolor{gray} 33 &    &    &    &    &    & \\
   & 00 & 10 & 20 & 30 & 31 & \cellcolor{gray} 32 &\cellcolor{gray}  33 &    &    &    &    & \\
   &    &  00 &  10 &  20 &  21 & \cellcolor{gray} 22 &\cellcolor{gray}  32 & \cellcolor{gray} 33 &    &    &    & \\
   &    &    & \cellcolor{gray} 00 & \cellcolor{gray} 10 & \cellcolor{gray} 11 & \cellcolor{gray} 21 & \cellcolor{gray} 31 & \cellcolor{gray} 32 &\cellcolor{gray} 33 &    &    & \\
   &    &    &    & \cellcolor{gray} 00 &\cellcolor{gray}  10 &\cellcolor{gray}  20 & 30 & 31 & 32 & 33 &    & \\
   &    &    &    &    &\cellcolor{gray} 00 &\cellcolor{gray} 10 & 20 & 21 & 22 & 32 & 33 & \\
   &    &    &    & & & \cellcolor{gray} 00 & 10 & 11 & 21 & 31 & 32 & 33 \\
\end{array}$$
\caption{Crossing a wall}\label{wall example}
\end{figure}

The top half of Figure~\ref{wall example} is the cylindrical growth diagram of Figure~\ref{growth example}. 
The bottom half shows the new growth diagram we obtain when we cross a wall reversing the order of $(s(4), s(5), s(6))$. 
Proposition~\ref{wall cross} tells us that the southwest shaded triangle in the bottom diagram is identical to that in the top diagram, and that the northeast shaded triangles are reflections of each other over a line of slope $1$. (This reflection interchanges the entries $22$ and $31$.) The unshaded entries must then be computed recursively.

We conclude with the promised conceptual verification that the points~(3) and~(6) in Theorem~\ref{basic monodromy} are ordered as claimed.

\begin{lemma} \label{jdt domino}
Let $\gamma_{\bullet \bullet}$ be a cylindrical growth diagram. As we go from $\gamma_{ij}$ to $\gamma_{i(j+1)}$ to $\gamma_{i(j+2)}$, we add two boxes. 
If the second box added is northeast of the first, then $\gamma_{j(j+2)} = (2)$. If the second box is added southwest of the first, then $\gamma_{j(j+2)} = (1,1)$.
\end{lemma}

\begin{proof}
Suppose that the second box is added northeast of the first. Then I claim that, for every $k$ between $j-r+2$ and $j$, when we go from $\gamma_{kj}$ to $\gamma_{k(j+1)}$ to $\gamma_{k(j+2)}$, the second box is added northeast of the first.
This is just checked by looking what the recursion does when we change from the path $(\gamma_{(k+1)j}, \gamma_{kj}, \gamma_{k(j+1)}, \gamma_{k(j+2)})$ to $(\gamma_{(k+1)j}, \gamma_{(k+1)(j+1)}, \gamma_{(k+1)(j+2)}, \gamma_{k(j+2)})$.

In particular, along the path $(\gamma_{jj}, \gamma_{j(j+1)}, \gamma_{j(j+2)})$, we add the second box northeast of the first. But $\gamma_{jj} = \emptyset$ and $\gamma_{j(j+1)} = \square$, so the only way to add a square to the northeast is for $\gamma_{j(j+2)}$ to be $(2)$, a horizontal domino. 
Similar arguments do the other case.
\end{proof}

This is a very easy case of \emph{jeu de taquin}, which tells us how to transform $(\gamma_{ij}, \gamma_{i(j+1)}, \ldots, \gamma_{i(j+m)})$ to $(\gamma_{jj}, \gamma_{j(j+1)}, \ldots, \gamma_{j(j+m)})$. 

\begin{proof}[Conceptual verification of the missing point from Theorem~\ref{basic monodromy}] 
For notational simplicity, we work over the face of $\Mbar{r}$ corresponding to the circular ordering $1$, $2$, \dots, $r$. 

Take any cylindrical growth diagram with $(\gamma_{ij}, \gamma_{i(j+1)}, \gamma_{i(j+2)}) = (\lambda, \kappa_1, \mu')$. (Here $i$ and $j$ are chosen such that $j-i = |\lambda|$; these are not the same $i$ and $j$ as in Theorem~\ref{basic monodromy}.) 
Consider a stable curve with three components: a central component with markings $j$ and $j+1$, and two other components joined to this one, one with marked points indexed by $[i,j)$ and the other with marked points labeled by $[j+2, i)$. 
Let the component marked by $[i,j)$ be attached to the central component at $0$; 
let the marked point $j$ be $\tau$; let the marked point $j+1$ be $1$;
and let the component marked by $[j+2, i)$ be attached at $\infty$, with $0 < \tau < 1 < \infty$.
When $\tau$ collides with $0$, a new component bubbles off and, where that new component is attached to the central component, we see the marking $\kappa_1$.
When $\tau$ collides with $1$, a new component bubbles off and, where that new component is attached to the central component, we see $\gamma_{j(j+2)}$ which, by Lemma~\ref{jdt domino}, is $(2)$.

If we let $\tau$ go all the way around the central component, we recover the monodromy from Theorem~\ref{basic monodromy}. 
In particular, we have just seen that, as $\tau$ goes from $0$ to $1$, the point above $0$ where the node is labeled by $\kappa_1$ goes to $(2)$.
Similarly, the point above $0$ where the node is labeled by $\kappa_2$ goes to $(1,1)$.
\end{proof}

\subsection{The case $d=2$ and the work of Eremenko and Gabrielov}

In this section, we consider the special case $d=2$. The number of standard Young tableaux of shape $2 \times (n-2)$ is the Catalan number $C_n = \frac{(2n-4)!}{(n-2)!(n-1)!}$, and so this is also the number of cylindrical growth diagrams.
The bijection between cylindrical growth diagrams and standard Young tableaux requires choosing a path through $\II$.

Consider a sequence of points $p_1$, $p_2$, \dots, $p_{2n-4}$ arranged around a circle; note that $2n-4=r$. A \newword{non-crossing matching} is a partition of $[r]$ into $n-2$ sets of size 2 such that, if $\{ i,j \}$ and $\{ k, \ell \}$ are elements of the matching, then the chords $\overline{p_i p_j}$ and $\overline{p_k p_{\ell}}$ don't cross.
$C_n$ also counts the number of noncrossing matchings. 

The following is a bijection from noncrossing matchings to cylindrical growth diagrams: 
Among $\{ p_i, p_{i+1}, \ldots, p_{j-1} \}$, let there be $s$ arcs connecting one of these points to another, and let there be $t$ arcs connecting one of the points to a point in $\{ p_j, p_{j+1}, \ldots, p_{i-1} \}$, so $j-i = 2s+t$. Take $\gamma_{ij} = (s+t, s)$. 

This bijection was published in~\cite{PPR}, where it is credited to White. More precisely,~\cite{PPR} describes this as a bijection between noncrossing matchings and standard young tableaux which carries the rotational symmetry of matchings to the action of promotion on tableaux. A cylindrical growth diagram records all of the promotions of a given tableaux.

Observe that there is an arc of the matching from $p_i$ to $p_j$ if and only if, for some $s$, we have $\gamma_{(i+1)j} = (s,s)$ and $\gamma_{i(j+1)}=(s+1, s+1)$.
So $\gamma$ is determined by knowing the set of $(i,j)$ for which $((\gamma_{(i+1)j}, \gamma_{i(j+1)})$ is of the form $((s,s), (s+1, s+1))$.
The reader may like to verify that the growth diagram of Figure~\ref{growth example} corresponds to the matching $\{ 1, 2 \}$, $\{ 3, 4 \}$, $\{ 5, 6 \}$.

We introduce the abbreviation $\cS_0$ for the part of $\cS(\RR)$ lying above $\Mzero{r}(\RR)$ (note the absence of a bar).
Eremenko and Gabrielov~\cite{EG} describe a beautiful way to decompose $\cS_0$ into cells labeled by noncrossing matchings. We will show that our decomposition corresponds to theirs under the above combinatorial recipe.

We now explain the construction of~\cite{EG}; it begins by using the connection to the Wronski problem discussed in the introduction.
Take a point $w$ in $\cS_0$, and let $z_1$, $z_2$, \dots, $z_r \in \RR \PP^1$ be the marked points of the point of $\Mzero{r}(\RR)$ below $w$.
Identify $\RR^n$ with $\RR[x]_{\leq n-1}$, the space of polynomial of degree $\leq n-1$; let $w$ be spanned by $p(x)$ and $q(x) \in \RR[x]_{\leq n-1}$.
Let $\phi: \PP^1 \to \PP^1$ be the rational map $z \mapsto p(z)/q(z)$; choosing a different basis for $w$ changes this map by an automorphism of the target. The Schubert condition $\Omega(\square, z_i)$ means that $\phi$ has a critical point at $z_i$.
Let $W :=\phi^{-1}(\RR \PP^1)$, a one dimensional CW complex embedded in $\CC \PP^1$. Let $N$ and $S$ be the two hemispheres into which $\CC \PP^1$ is divided by $\RR \PP^1$. 

By the Shapiro-Shapiro conjecture (which was proved in~\cite{EG} in this case), $\phi$ is defined over $\RR$. This implies that $\RR \PP^1 \subset W$ and that $W \cap N$ and $W \cap S$ are complex conjugates of each other.
$W \cap N$ is a collection of non-crossing arcs, terminating at the $z_i$. (One might also imagine that $W$ contained some closed loops disjoint from the $z_i$, but Eremenko and Gabrielov prove that it doesn't.)
The main result of~\cite{EG} is that there, for any $(z_1, z_2, \ldots, z_r)$ in $\Mzero{r}(\RR)$, and any noncrossing matching, there is precisely one element of $\cS_0$ where $W$ has the topology given by that matching.

\begin{theorem} \label{eg match}
Let $\tau$ be the maximal face of $\cS$ containing $w$; let $\gamma$ be the cylindrical growth diagram indexing $\tau$; let $M$ be the corresponding noncrossing matching.
Then $N \cap W$ has the same topology as $M$.
\end{theorem}

\begin{proof}
We will show that, if $((\gamma_{(i+1)j}, \gamma_{i(j+1)})$ is of the form $((s,s), (s+1, s+1))$, then $W$ contains an arc joining $z_i$ and $z_j$. By the previous observations, there is $\gamma_{ij}$ is of the form $(s,s)$ if and only if there is an arc of $M$ from $i$ to $j$, and we know that $M$ and $N \cap W$ both have $r/2 =n-2$ arcs, so this will show that $M$ and $N \cap W$ have the same connectivity.

So, let $((\gamma_{(i+1)j}, \gamma_{i(j+1)})$ be of the form $((s,s), (s+1, s+1))$. Note that the topology of $W$ is constant on connected components of $\cS_0$. So we can let $w$ approach the codimension $2$ face of $\tau$ corresponding to stable curves with three components: one which contains $\{ z_k: k \in (i,j) \}$, one which contains $\{ z_\ell: \ell \in (j,i) \}$ and one, the middle component, which contains $z_i$ and $z_j$. Call these components $C_1$, $C_2$, $C_3$.
In the limit where $w$ lies in this codimension $2$ face, we can consider $w$ as a point in $G(2,n)_{C_1} \times G(2,n)_{C_2} \times G(d,n)_{C_3}$, as in Section~\ref{sec families}. 

Let $w_2$ be the $C_2$ component of $w$. As above, let $w_2$ be spanned by $p_2$ and $q_2 \in \RR[x]_{\leq n-1}$.
Choose coordinates on $C_2$ such that the node joining $C_1$ and $C_2$ is at $0$, and the node joining $C_2$ and $C_3$ at $\infty$. 
The condition $\gamma_{(i+1)j}=(s,s)$ means that $p_2$ and $q_2$ vanish to order $s$ at $0$; the condition $\gamma_{i(j+1)}=(s+1, s+1)$ means that $\deg p_2=\deg q_2 =s+2$. So $\phi_2=p_2/q_2$ is the unique degree $2$ rational map with critical points $z_i$ and $z_j$. 
For this rational function, $\phi_2^{-1}(\RR \PP^1)$ is four arcs joining $z_i$ and $z_j$; two of them passing through $0$ and $\infty$.

When we move $w$ off of the codimension $2$ face, so that the node smooths, the two arcs from $z_i$ to $z_j$ which avoid the nodes remain.
So $W$ has an arc from $z_i$ to $z_j$, as desired.
\end{proof}

\begin{remark}
One might hope for a similar construction in the case $d=3$. Kuperberg~\cite{Kup} defines $A_2$ webs, the number of which is the same as the number of standard Young tableaux of shape $3 \times (n-3)$; these $A_2$ webs have rotational symmetry, just as noncrossing matchings do.
Khovanov-Kuperberg~\cite{KK} present a bijection from webs to Young tableaux. Translated into a bijection from webs to cylindrical growth diagrams, the map of~\cite{KK} is the following: Take a minimal cut path $\delta$ from $i$ to $j$, as defined in~\cite[Section 5]{KK}. Then $\gamma_{ij} = (s+t+u, s+t, s)$ where $s$ is the number of arcs crossing $\delta$ from left to right, $u$ is the number of arcs crossing $\delta$ from right to left, and $t$ is determined by the condition $j-i = 3s+2t+u$.
It would be fascinating to find a geometric interpretation of webs similar to the interpretation Eremenko and Gabrielov find for non-crossing matchings.
\end{remark}

\section{Dual equivalence} \label{sec:dual}

In the last section, we gave an explicit indexing set for the facets of $\cS(\square, \square, \cdots, \square)$ lying over a fixed facet of $\Mbar{r}$.
In the next section, we will generalize this to an arbitrary Schubert problem $\cS(\lambda_1, \lambda_2, \dots, \lambda_r)$ with $\sum |\lambda_i| = d(n-d)$. 
Our generalization uses the language of dual equivalence, so we rapidly review this topic.
Our reference is~\cite{Haiman}

Let $\JJ$ be a a subset of $\ZZ^2$ of the form $[a,b] \times [c,d]$. 
A \newword{growth diagram} is a map $\gamma: \JJ \to \Lambda$ which satisfies conditions~(1) and~(2) in the definition of a cylindrical growth diagram. In other words, a growth diagram is a possible rectangular region within a cylindrical growth diagram.
(Historically, growth diagrams are the standard notion; the term ``cylindrical growth diagram" is due to this paper.)

Note that, if we know the partitions $\gamma_{ij}$ for $(i,j) \in \{ b \} \times [c,d] \cup [a,b] \times \{d \}$, then all of the others are recursively determined.
We introduce terminology for this idea: Let $\mu \subset \nu \subset \pi$ be a chain of partitions, let $\delta_{\bullet} \in \SYT(\nu/\mu)$ and $\epsilon_{\bullet} \in \SYT(\pi/\nu)$. 
Set $ m = |\mu|$, $n = |\nu|$ and $p = |\pi|$, so $\delta$ and $\epsilon$ are chains of length $n-m$ and $p-n$.
Let $\gamma: [-(p-n),0] \times [m,n] \to \Lambda$ be the unique growth diagram
with $\gamma_{0(m+j)} = \delta_j$ and $\gamma_{(-i)n} = \epsilon_i$.
Let $\delta'$ be the tableau $\delta'_j =\gamma_{(n-p) (m+j)}$ and let $\epsilon'$ be  the tableau $\epsilon'_i = \gamma_{(-i)0}$.
So $\delta' \in \SYT(\pi/\nu')$ and $\epsilon' \in \SYT(\nu'/\mu)$ for some $\nu'$. 
We say that $(\epsilon', \delta')$ is the result of \newword{shuffling} $(\delta, \epsilon)$.
The growth diagram recursion is symmetric in switching $-i$ and $j$, so shuffling $(\epsilon', \delta')$ restores the original $(\delta, \epsilon)$.

We define two equivalence relations called \newword{slide equivalence} and \newword{dual equivalence}.

\newword{Slide equivalence} is the finest equivalence relation satisfying the following: If $(\delta, \epsilon)$ shuffles to $(\epsilon', \delta')$, then $\delta$ is slide equivalent to $\delta'$ and $\epsilon$ is slide equivalent to $\epsilon'$. 

Two dual equivalent tableaux will always have the same shape. \newword{Dual equivalence} is the coarsest equivalence relation on $\SYT(\nu/\mu)$ satisfying the following: Let $\delta_1$ and $\delta_2 \in \SYT(\nu/\mu)$ be dual equivalent, and let $\epsilon$ be any element of $\SYT(\pi/\nu)$, for any $\pi \supseteq \nu$. Write $(\epsilon'_1, \delta'_1)$ and $(\epsilon'_2, \delta'_2)$ for the results of shuffling $\epsilon$ past $\delta_1$ and $\delta_2$. Then $\epsilon'_1 = \epsilon'_2$.  
Conceptually, ``shuffling past two dual equivalent tableaux has the same effect".

The following are the key results about dual and slide equivalence:

\begin{prop}[{\cite[Theorem 2.13]{Haiman}}] \label{thing1}
If two tableaux are both dual and slide equivalent, then they are equal.
\end{prop}

\begin{prop}[{\cite[Corollary 2.5]{Haiman}}]  \label{thing2}
Any two tableaux in $\SYT(\mu/\emptyset)$ are dual equivalent to each other. 
\end{prop}

\begin{prop} \label{thing3}
If $\delta_1$ and $\delta_2$ are dual equivalent, and $(\epsilon', \delta'_1)$, $(\epsilon', \delta'_2)$ are the results of shuffling the $\delta_i$ with some $\epsilon$, then $\delta'_1$ and $\delta'_2$ are dual equivalent.
\end{prop}

\begin{proof}
Immediate corollary of \cite[Lemma 2.3 and Theorem 2.6]{Haiman}. 
\end{proof}

\begin{prop} \label{thing4}
Any slide equivalence class contains a unique tableaux $\delta$ of straight shape.
\end{prop}

\begin{proof}
This is the ``first fundamental theorem of \emph{jeu de taquin}''. As Haiman sketches in the final paragraphs of~\cite[Section 2]{Haiman}, this follows from Proposition~\ref{thing2}.
\end{proof}

 Given an arbitrary tableaux $\epsilon$, we will say that $\epsilon$ \newword{rectifies} to $\delta$ if $\delta$ is the unique tableaux of straight shape in its slide equivalence class. In this context, we will write $\rsh(\epsilon)$ for $\sh(\delta)$. If $\epsilon_1$ and $\epsilon_2$ are dual equivalent, then $\rsh(\epsilon_1) = \rsh(\epsilon_2)$.

\begin{prop} \label{thing5}
Given $\epsilon \in \SYT(\mu/\emptyset)$, and $\delta$ a tableaux with $\rsh(\delta) = \mu$, there is a unique tableaux $\phi$ which is slide equivalent to $\epsilon$ and dual equivalent to $\delta$.
\end{prop}

\begin{proof}
Uniqueness is an immediate consequence of Proposition~\ref{thing1}; we must show existence. 

Let $\sh(\delta) = \nu/\lambda$. Choose $\alpha \in \SYT(\lambda)$. Let $(\epsilon_2, \beta)$ be the result of shuffling $(\alpha, \delta)$. 
So $\sh(\epsilon_2) = \rsh(\delta) = \mu$ and thus $\sh(\beta) = \nu/\mu$; we see that it makes sense to shuffle $(\epsilon, \beta)$.
Let the result of shuffling $(\epsilon, \beta)$ be $(\psi, \phi)$.
By definition, $\epsilon$ is slide equivalent to $\phi$.
By Proposition~\ref{thing2}, $\epsilon$ and $\epsilon_2$ are dual equivalent so, by Proposition~\ref{thing3}, $\delta$ and $\phi$ are dual equivalent.
\end{proof}

We now introduce shuffling of dual equivalence classes.

\begin{prop} \label{de shuff}
Let $\delta \in \SYT(\nu/\mu)$ and let $\epsilon \in \SYT(\pi/\nu)$, for some $\mu \subset \nu \subset \pi$. Let $(\delta, \epsilon)$ shuffle to $(\epsilon', \delta')$.
The dual equivalence classes of $\delta'$ and $\epsilon'$ are determined by the dual equivalence classes of $\delta$ and $\epsilon$.
\end{prop}

\begin{proof}
Let $\epsilon_1$ and $\epsilon_2$ be dual equivalent. Let $(\epsilon'_1, \delta')$ and $(\epsilon'_2, \delta')$ be the results of shuffling $(\delta, \epsilon_1)$ and $(\delta, \epsilon_2)$. (The second entry is the same in both cases by the definition of dual equivalence.) Then $\delta'$ is tautologically dual equivalent to itself, and $\epsilon'_1$ and $\epsilon'_2$ are dual equivalent by Proposition~\ref{thing3}. 

So replacing $\epsilon$ by a dually equivalent tableaux does not change the dual equivalence classes of $\delta'$ and $\epsilon'$. 
Similarly, replacing $\delta$  by a dually equivalent tableaux does not change the dual equivalence classes of $\delta'$ and $\epsilon'$. 
Combining these two facts, we have the claim.
\end{proof}

Thus, we may speak of \newword{shuffling two dual equivalence classes}, to obtain another pair of dual equivalence classes.

\section{Dual equivalence classes of cylindrical growth diagrams} \label{sec:degrowth}

Fix partitions $\lambda_1$, $\lambda_2$, \dots, $\lambda_r$ obeying $\sum_{k=1}^r |\lambda_k| = d(n-d)$. (We no longer are setting $r=d(n-d)$.) Let $\II$, as in section~\ref{sec:growth}, be $\{ (k,\ell) \in \ZZ^2 : k \leq \ell \leq k+r \}$. 

A \newword{dual equivalence cylindrical growth diagram}, which we will abbreviate to \newword{decgd}, consists of two pieces of data:
\begin{enumerate}
\item A map $\gamma: \II \to \Lambda$, such that $\gamma_{k\ell}$ is contained in $\gamma_{k(\ell+1)}$ and $\gamma_{(k-1)\ell}$ and
\item For every $(k,\ell) \in \II$ with $\ell-k < r$, a choice $a(k,\ell)$ and $b(k, \ell)$ of dual equivalence classes in $\SYT(\gamma_{k(\ell+1)}/\gamma_{k\ell})$ and $\SYT(\gamma_{(k-1)\ell}/\gamma_{k\ell})$ respectively.
\end{enumerate}
such that
\begin{enumerate}
\item $\gamma_{kk} = \emptyset$ and $\gamma_{k(k+r)} = \rect$.
\item If we shuffle the dual equivalence classes $a(k,\ell)$ and $b(k, \ell+1)$, we obtain the dual equivalence classes $b(k,\ell)$ and $a(k-1,\ell)$.
\end{enumerate}
We will use indices $(k, \ell)$ for dual equivalence cylindrical growth diagrams, and reserve $(i,j)$ for ordinary cylindrical growth diagrams.
Note that, since the shape of $a(k,\ell)$ is $\gamma_{k(\ell+1)}/\gamma_{k\ell}$, the dual equivalence classes $a(k, \ell)$ and $b(k, \ell)$ determine the map $\gamma$.

The \newword{shape} of the decgd $\gamma$ is the sequence of partitions $(\gamma_{01}, \gamma_{12}, \ldots, \gamma_{(r-1)r})$.

The main result of this section is that, over the face of $\Mbar{r}(\RR)$ associated to the circular ordering $(s(1),\ldots, s(r))$, the maximal faces of $\cS(\lambda_1, \ldots, \lambda_r)(\RR)$ are indexed 
by decgds of shape $(\lambda_{s(1)}, \ldots, \lambda_{s(r)})$. 

Intuitively, we want decgds to be equivalence classes of cylindrical growth diagrams. We now explain how to map a cylindrical growth diagram $\gamma$ to a decgd, an operation we call \newword{restricting} $\gamma$.

Set $\tr = d(n-d)$ and $\tII = \{ (i,j)  \in \ZZ^2 : i \leq j \leq i+ \tr \}$. 
Let $\tg: \tII \to \Lambda$ be a cylindrical growth diagram. We will explain how to use $\tg$ to build a decgd which we will call the \newword{restriction of $\tg$}.
Let $\iota: \ZZ \to \ZZ$ be such that $\iota(\ell)- \iota(\ell-1) = |\lambda_{\ell}|$, with the index $\ell$ periodic modulo $r$ on the right hand side.
So $\iota(i+r) - \iota(i) = \sum_{\ell=1}^r |\lambda_{\ell}| = \tr$. 
Let $\gamma_{k \ell} = \tg_{\iota(k) \iota(\ell)}$; let $a(k,\ell)$  be the dual equivalence class of $\tg_{\iota(k) \iota(\ell)}$, $\tg_{\iota(k) (\iota(\ell)+1)}$, $\tg_{\iota(k) (\iota(\ell)+2)}$, \dots, $\tg_{\iota(k) \iota(\ell+1)}$; and let $b(k,\ell)$ be the dual equivalence class of $\tg_{\iota(k) \iota(\ell)}$, $\tg_{(\iota(k)-1) \ell}$, \ldots, $\tg_{\iota(k-1) \iota(\ell)}$.


\begin{prop} \label{decgd basics}
Every decgd is the restriction of a cylindrical growth diagram. More specifically, choose a path $\delta$ through $\II$ and choose lifts $\tilde{a}(k,\ell)$ and $\tilde{b}(k, \ell)$ of $a(k, \ell)$ and $b(k, \ell)$ along $\gamma$; then there is a unique way to extend this lifting to cylindrical growth diagram. The number of decgd's of shape $(\lambda_1, \lambda_2, \ldots, \lambda_r)$ is the Littlewood-Richardson coefficient $c_{\lambda_1 \lambda_2 \cdots \lambda_r}^{\rect}$.
\end{prop}

\begin{proof}
Choose a path $\delta$ through $\II$ and choose lifts $\tilde{a}(k, \ell)$, $\tilde{b}(k, \ell)$ of $a(k, \ell)$ and $b(k, \ell)$ along this path.
By Lemma~\ref{recursion}, these lifts extend uniquely to a cylindrical growth diagram $\tilde{\gamma}$; what we need to know is that $\tilde{\gamma}$ restricts to the original $\gamma$. But this is a simple inductive statement from the definition of shuffling dual equivalence classes.

We now establish the claim about the Littlewood-Richardson coefficient. For notational convenience, choose our starting path $\delta$ to be $(0,0) \to (0,1) \to (0,2) \to \cdots \to (0,r)$.
Abbreviate $\gamma_{0j}$ to $\mu_j$ and $a(0,j)$ to $a(j)$. As in Proposition~\ref{recursion}, there is a unique $\gamma$ extending any choice of partitions $\mu_j$ and dual equivalence classes $a(j) \in \SYT(\mu_{j+1}/\mu_j)$. This unique $\gamma$ will have shape $(\lambda_1, \lambda_2, \ldots, \lambda_r)$ if and only if $\rsh(a(j)) = \lambda_j$. The number of dual equivalence classes in $ \SYT(\mu_{j+1}/\mu_j)$ with rectification shape $\lambda_{j+1}$ is $c_{\lambda_{j+1} \mu_j}^{\mu_{j+1}}$ (see~\cite[Theorem A.1.3.1]{EC2} and Proposition~\ref{thing5}.) So the number of decgd's of shape $(\lambda_1, \lambda_2, \ldots, \lambda_r)$ is 
$$\sum_{\emptyset =\mu_0 \subset \mu_1 \subset \mu_2 \subset \cdots \subset \mu_r=\rect} c_{\mu_0, \lambda_1}^{\mu_1} c_{\mu_1, \lambda_2}^{\mu_2} \cdots c_{\mu_{r-1}, \lambda_r}^{\mu_r} =c_{\lambda_1 \lambda_2 \cdots \lambda_r}^{\rect}$$
by the associativity of the Littlewood-Richardson coefficients.
\end{proof}

We now describe the indexing of maximal faces of $\cS(\lambda_1, \ldots, \lambda_r)(\RR)$ by decgds.
Embed $\Mbar{r}(\RR) \times \prod \Mbar{|\lambda_k|+1}(\RR)$ into $\Mbar{\tr}(\RR)$ by taking $(C, C_1, \ldots, C_r) \in \Mbar{r}(\RR) \times \prod \Mbar{|\lambda_k|+1}(\RR)$ and sending it to the stable curve in $\Mbar{\tr}(\RR)$ where we glue the $(|\lambda_k|+1)$th marked point of $C_k$ to the $k$th marked point of $C$.

Look at the family $\cS(\square, \square, \ldots, \square)$ over $\Mbar{\tr}(\RR)$. Restrict it to a family $\cX$ over $\Mbar{r}(\RR) \times \prod \Mbar{|\lambda_k|+1}(\RR)$. 
In the family $\cX$, the markings on the $r$ nodes of the central component are locally constant. 
Let $\cY$ be those components of $\cX$ where the markings on the central nodes are $(\lambda_1, \lambda_2, \ldots, \lambda_r)$.

Take the face $\sigma$ of $\Mbar{r}$ corresponding to the circular ordering $(1,2,\ldots, r)$. 
Let $\widetilde{\sigma}$ be some maximal face of $\Mbar{\tr}$ containing $\sigma \times  \prod \Mbar{|\lambda_k|+1}(\RR)$.
Let $\tau$ in $\cY$ lie over $\sigma$ and let $\widetilde{\tau}$ be the face over $\widetilde{\sigma}$ in $\cS(\square, \square, \cdots, \square)$ containing $\tau \times  \prod \Mbar{|\lambda_k|+1}(\RR)$ in its boundary.
Then $\tau$ is associated to some cylindrical growth pattern $\tg$. Restricting $\tg$ to $\II$ gives a decgd pattern $\gamma$ of shape $(\lambda_1, \ldots, \lambda_r)$. 

\begin{theorem} \label{dual labeling}
With the above notations, $\gamma$ depends only on $\tau$, not on the choice of $\widetilde{\tau}$. This gives a bijection between faces of $\cS(\lambda_1, \ldots, \lambda_r)$ over $\sigma$ and dual equivalence cylindrical growth patterns of shape $(\lambda_1, \ldots, \lambda_r)$. 
\end{theorem}

\begin{remark} 
We describe faces of $\cS(\lambda_1, \ldots, \lambda_r)$ as boundaries of faces of $\cS(\square, \square, \cdots, \square)$. This is closely related to Sections~5 and~6 of \cite{Purb1}, which describe a limiting procedure where the marked points on $\Mbar{d(n-d)}$ collide. 
Because we have available $\Mbar{r}$ as a base space, rather than $(\PP^1)^r$, we can see precisely where the different dual equivalent growth diagrams wind up in this limiting process -- in the different parts of the $ \prod \Mbar{|\lambda_k|+1}(\RR)$ factor.
\end{remark}

\begin{proof}
By Proposition~\ref{decgd basics}, we know that we can lift $\gamma$ to a cylindrical growth diagram $\tg$.
Therefore, we have shown that for every dual equivalence cylindrical growth diagram $\gamma$, there is a face $\widetilde{\tau}$ in $\cS(\square, \square, \ldots, \square)$ which will give $\gamma$.
Our next goal is to show that $\gamma$ depends only on $\tau$ (the face of $\cY$), not $\widetilde{\tau}$ (the face of $\cS(\square, \cdots, \square)$).

Let $\widetilde{\tau}$ and $\widetilde{\tau}'$ be two faces containing $\tau$; let $\tg$ and $\tg'$ be the corresponding cylindrical growth patterns.
Let $\widetilde{a}(k,\ell)$ be the tableau $(\tg_{\iota(k), \iota(\ell)},\linebreak[0] \tg_{\iota(k), \iota(\ell)+1},\linebreak[0] \ldots, \linebreak[0] \tg_{\iota(k), \iota(\ell+1)})$, let $\widetilde{b}(k,\ell)$ be the tableau $(\tg_{\iota(k), \iota(\ell)},\linebreak[0] \tg_{\iota(k)-1, \iota(\ell)},\linebreak[0] \ldots, \linebreak[0] \tg_{\iota(k-1), \iota(\ell)})$ and define the primed objects similarly.
We want to show that $\widetilde{a}(k,\ell)$ is dual equivalent to $\widetilde{a}'(k,\ell)$, and the same for $\widetilde{b}(k, \ell)$ and $\widetilde{b}'(k,\ell)$.

Abbreviate $\cS(\square, \square, \ldots, \square)$ to $\cS$. 
Let $\mathcal{N}$ be a tubular neighborhood of  $\tau  \times  \prod \Mbar{|\lambda_k|+1}(\RR)$ and let $\pi$ be the projection from $\mathcal{N}$ onto  $\tau  \times  \prod \Mbar{|\lambda_k|+1}(\RR)$.
Take a path $p(t)$ from $\widetilde{\tau}$ to $\widetilde{\tau}'$ through $\mathcal{N}$ which only passes through cells of codimension $0$ and $1$. We are using that $\prod \Mbar{|\lambda_k|+1}(\RR)$ is connected and $\cS$ is a manifold.
Each maximal cell through which $p(t)$ passes is labeled by a cylindrical growth diagram. We claim that all of these diagrams are dual equivalent.

There are two types of walls we cross through. One type of wall is when $\pi(p(t))$ crosses a wall within one of the factors $ \Mbar{|\lambda_m|+1}(\RR)$. 
The edge labels $\widetilde{a}(k,\ell)$ and $\widetilde{b}(k, \ell)$ are unchanged when $k \leq m-1 < m \leq \ell$. 
Look at the path $\widetilde{b}(m,m)$, $\widetilde{b}(m-1,m)$, \ldots, $\widetilde{b}(m-r,m)$. The tableaux after the first position are preserved; the tableau in the first position is replaced by another one of the same straight shape. So, by Proposition~\ref{thing2}, we have preserved dual equivalence classes along this path and hence in the whole growth diagram.

The other type of wall is one where $\pi(p(t))$ remains in the same cell of  $\tau  \times  \prod \Mbar{|\lambda_k|+1}(\RR)$. When we cross such a wall, we reverse the circular ordering of some interval of points of length $\lambda_m$, and the effect on the growth diagram is to switch $\widetilde{a}(m-1,m-1)$ with $\widetilde{b}(m, m)$, while preserving $\widetilde{a}(k,\ell)$ and $\widetilde{b}(k, \ell)$ for $k \leq m-1 < m \leq \ell$.
As before, the dual equivalence classes of all the tableaux  on the path $\widetilde{b}(m,m)$, $\widetilde{b}(m-1,m)$, \ldots, $\widetilde{b}(m-r,m)$ are unchanged.
Again, the dual equivalence class of the cylindrical growth diagram is unchanged in this case.


We have now shown that the map from faces $\tau$ to dual equivalence cylindrical growth diagrams is well defined. We must show that it is bijective.
Since we can always lift a decgd $\gamma$ to a cylidrical growth diagram $\tg$, the map is surjective.

The covering $\cY \to \Mbar{r}$ is of degree $c_{\lambda_1 \lambda_2 \cdots \lambda_r}^{\smallrect}$. This is also the number of dual equivalence cylindrical growth diagrams. So the correspondence between maximal faces and decgds is bijective.
\end{proof}

We can now describe the $CW$-structure on $\cS(\lambda_1, \lambda_2, \ldots, \lambda_r)$.
Its facets are indexed by pairs of a circular ordering of $[r]$ and a dual equivalence cylindrical growth diagram of shape $(\lambda_1, \ldots, \lambda_r)$.
Finally, we describe how these facets are glued together, generalizing Proposition~\ref{wall cross}.

Let $1 \leq p < q \leq r$ and suppose we cross a wall reversing the order of $s(q)$, $s(q+1)$, \dots, $s(p-1)$ (indices cyclic modulo $r$). 
Let $(\gamma, a, b)$ label the face on one side of the wall and $(\hat{\gamma}, \hat{a}, \hat{b})$ label the face on the other.

For $p \leq k \leq \ell \leq q$, we have $\gamma_{k\ell} = \hat{\gamma}_{k\ell}$; 
for $p \leq k \leq \ell < q$, we have $a(k, \ell) = \hat{a}(k, \ell)$; 
for $p < k \leq \ell \leq q$, we have $b(k, \ell) = \hat{b}(k, \ell)$.
For $q \leq k \leq \ell \leq p+r$, we have $\gamma_{k\ell} = \hat{\gamma}_{(p+q-\ell-1)(p+q-k-1)}$;
for $q \leq k \leq \ell < p+r$, we have $a(k,\ell) = \hat{b}(p+q-\ell-1, p+q-k-1)$ and $\hat{a}(k,\ell) = b(p+q-\ell-1, p+q-k-1)$, 

For proof, simply lift $\gamma$ to some $\tg$, cross the corresponding wall in $\cS(\square,\linebreak[0] \square,\linebreak[0] \cdots,\linebreak[0] \square)$, and apply Proposition~\ref{wall cross}.

\end{document}